\documentclass[12pt]{article}
\usepackage[T1]{fontenc}
\usepackage{amsmath,amssymb,amsthm}
\usepackage{booktabs}
\usepackage{float}
\usepackage{cite}
\usepackage{geometry}
\usepackage{microtype}
\usepackage{newtxtext,newtxmath}
\usepackage{bm}
\usepackage{hyperref}
\hypersetup{
  colorlinks=true,
  linkcolor=blue,
  citecolor=red,
  urlcolor=blue
}

\allowdisplaybreaks[4]
\newtheorem{Theorem}{Theorem}[section]
\newtheorem{Lemma}[Theorem]{Lemma}
\newtheorem{prop}[Theorem]{Proposition}
\newtheorem{coro}[Theorem]{Corollary}
\newtheorem{con}[Theorem]{Conjecture}

\newtheorem{claim}[Theorem]{Claim}

\theoremstyle{remark}
\newtheorem{remark}[Theorem]{Remark}

\numberwithin{equation}{section}

\begin{document}

\begin{center}
{\Large\bfseries
Covering the Hypercube $mB^n$
}\\ [7pt]
\end{center}

\begin{center}
    Zihao Huang$^{1}$,
    Miao Wang$^{*2}$
    and Suijie Wang$^{3}$\\[8pt]
    $^{1,2,3}$School of Mathematics\\
    Hunan University\\
    Changsha 410082, Hunan, P. R. China\\[12pt]
    
    Emails: $^{1}$zihaoh@hnu.edu.cn, $^{2}$miaowangmath@hnu.edu.cn, $^{3}$wangsuijie@hnu.edu.cn\\[3pt]
    $^*$Corresponding author.\\[12pt]
    
\end{center}

\vskip 3mm
\begin{abstract}
The Alon--Füredi theorem determines the minimum number of hyperplanes needed to cover every nonzero vertex of the Boolean cube while avoiding the origin. We study hyperplane coverings with multiplicity for the generalized hypercube
$mB^n=\{0,1,\ldots,m\}^n$. Our main result is a polynomial multiplicity theorem extending the theorem of Sauermann and Wigderson to $mB^n$, with sharp degree bounds for polynomials having prescribed high-order zeros. We apply this result to obtain lower and upper bounds for hyperplane coverings with multiplicity and determine the exact minimum number of hyperplanes for coverings with multiplicities $k=1,2$. In the special case of the Boolean cube, our constructions improve the upper bound of Clifton and Huang for certain ranges of the parameters.

    \vskip 6pt

    \noindent
    \textbf{Mathematics Subject Classification (2020):}
    05B40 (primary); 05D40, 05E05, 52C17 (secondary)
    \\ [7pt]
    \textbf{Keywords:}
 Hyperplane coverings with multiplicity; Boolean cube;
 Combinatorial Nullstellensatz; polynomial method;
 multiplicity of polynomial zeros; symmetric polynomials
    
\end{abstract}

\section{Introduction}
    Alon's Combinatorial Nullstellensatz \cite{Alon1} is one of the most powerful algebraic tools in modern combinatorics. The most classical application of the Combinatorial Nullstellensatz, due to Alon and F\"uredi \cite{Alon2}, is to solve a problem of Komj\'ath \cite{Komjath} that asks for the minimum number of hyperplanes in $\mathbb{R}^n$ that cover every point of the Boolean cube $B^n=\left\{0,1\right\}^n$ except the origin $\bm{0}$.

\begin{Theorem}[Alon and F\"uredi, 1993]\label{Alon2}
 For any polynomial $P\in \mathbb{R}[x_1,\ldots,x_n]$, if $P(\bm{0}) \ne 0$ and $P$ vanishes at every point of $B^n\setminus\{\bm{0}\}$, then $\deg P\ge n$.
\end{Theorem}
    Later research has focused on algebraic generalizations of Alon's Combinatorial Nullstellensatz to higher zero multiplicities \cite{Ball,Batzaya,Kos1,Kos2,Sauermann}, and geometric extensions of the Alon--F\"uredi lower bound to higher covering multiplicities \cite{Clifton,Sauermann}. Hyperplane coverings of Hamming layers and more general symmetric subsets of the Boolean cube have also been studied, both with and without multiplicity \cite{Venkitesh,GhoshKayalNandi,GhoshKayalNandiVenkitesh}. Combinatorial applications of higher-multiplicity results include Stepanov's method \cite{Hanson} and multiplicity versions of the Schwartz--Zippel lemma \cite{Bishnoi,Bukh,Dvir}.
    
Let $f(B^n,k)$ be the minimum size of a multiset of hyperplanes needed to cover every point of $B^n \setminus \{\bm{0}\}$ at least $k$ times while leaving $\bm{0}$ uncovered; repetitions are counted with multiplicity. The Alon--F\"uredi theorem gives that $f(B^n,1)=n$ and $f(B^n,2)=n+1$. Ball and Serra's Punctured Combinatorial Nullstellensatz \cite{Ball} gives the general lower bound $f(B^n,k)\ge n+k-1$. For $k\ge 4$ and $n\ge 3$, Clifton and Huang \cite{Clifton} improved this lower bound to $n+k+1$.
\begin{Theorem}[Clifton and Huang, 2020]\label{Clifton}
    For $n\ge 2$,
    \begin{equation*}
        f(B^n,3)=n+3.
    \end{equation*}
    For $k\ge 4$ and $n\ge 3$,
    \begin{equation*}
        n+k+1\le f(B^n,k)\le n+\binom{k}{2}.
    \end{equation*}
\end{Theorem}

    A new breakthrough was then made by Sauermann and Wigderson \cite{Sauermann}, who bypassed the geometric motivation and resolved the corresponding polynomial problem directly.
\begin{Theorem}[Sauermann and Wigderson, 2022]\label{Sauermann}
For $k\ge 2$ and $0\le l\le k-1$, let $P\in\mathbb{R}[x_1,\ldots,x_n]$ be a polynomial that has zeros of multiplicity at least $k$ at all points in $B^n \setminus \{\bm{0}\}$, and a zero of multiplicity exactly $l$ at $\bm{0}$. If $l=k-1$, then
        \[
        \deg P\ge n+2k-2.
        \]
     If $0\le l\le k-2$ and $n\ge 2k-3$, then
        \[
        \deg P \ge n+2k-3.
        \]
    In both cases the bound is attained for every $l$ in the stated range.
\end{Theorem}
    As a byproduct of Theorem \ref{Sauermann}, Sauermann and Wigderson derived the best known general lower bound, namely $f(B^n,k)\ge n+2k-3$ for $k\ge 4$ and $n\ge 2k-3$. Clifton and Huang \cite{Clifton} established the general upper bound $f(B^n,k)\le n+\binom{k}{2}$ and conjectured that it is tight if $n$ is sufficiently large with respect to $k$. It remains an interesting open problem to determine the exact value of $f(B^n,k)$ for $k\ge 4$.  

    In this paper, we consider a natural extension of their problem in which the Boolean cube is replaced by the general hypercube $mB^n=\{0,1,\ldots,m\}^n$, where $m,n$ are positive integers throughout. Let $f(mB^n,k)$ denote the minimum size of a multiset of hyperplanes that covers each point of $mB^n \setminus \{\bm{0}\}$ at least $k$ times while leaving $\bm{0}$ uncovered. Related multiplicity-covering problems have been studied for planar Cartesian grids \cite{BishnoiGrid}, subspaces over finite fields \cite{BishnoiSubspace}, triangular grids \cite{BasitCliftonHorn}, and half-grids \cite{BishnoiNene}.
    
    As a generalization of Theorem \ref{Clifton}, we obtain upper and lower
    bounds for $f(mB^n,k)$ in Theorem \ref{main-1}.  The upper bound is given
    by a family of constructions indexed by $c$. For $c=1,\ldots,k$, set
    $\rho_c=\min\{n,\lfloor(k-1)/c\rfloor\}$. Define
    \[
    U_{m,n,k}
    =m\min_{1\le c\le k}
    \left\{cn+\rho_ck-\frac{c\rho_c(\rho_c+1)}2\right\}.
    \]
     When $m=1$ and $n\ge k-1$, the choice $c=1$ gives
     $U_{1,n,k}\le n+\binom{k}{2}$, which is exactly the general upper bound
     of Clifton and Huang in Theorem \ref{Clifton}.  Since $U_{m,n,k}$ is
     defined by minimizing over all $c$, this construction can improve that
     bound.  For example, when $(m,n,k)=(1,3,4)$, the choice $c=2$ gives
     $U_{1,3,4}=8$, while the choice $c=1$ gives
     $3+\binom{4}{2}=9$.  This agrees with the exact value
     $f(B^3,4)=8$ proved by Clifton and Huang \cite{Clifton}.
\begin{Theorem}\label{main-1}
    For $n\ge 2$,
        \begin{equation*}
             f(mB^n,1)=mn\quad \text{and} \quad f(mB^n,2) = mn + m.\tag{1}  \label{eq1} 
        \end{equation*}
    For $k\ge 3$ and $n\ge k-1$,
        \begin{equation*}
        mn+(m+1)(k-1)-1\le f(mB^n,k) \le U_{m,n,k}. \tag{2} \label{eq2} 
        \end{equation*}
\end{Theorem}

The equalities in \eqref{eq1} are established by explicit hyperplane
constructions and matching lower bounds. Proposition
\ref{prop:parameterized-upper} gives the upper bound
$f(mB^n,k)\le U_{m,n,k}$, which in fact holds for all positive
$m,n,k$. For $k\ge3$ and $n\ge k-1$, Proposition
\ref{prop:c-one-threshold} shows that the choice $c=1$ is optimal if and only if
$n\ge\lfloor k^2/4\rfloor$. The lower bound in \eqref{eq2} follows
from Theorem \ref{main-2}, which extends Theorem \ref{Sauermann} from
$B^n$ to $mB^n$.

The equalities $f(mB^n,k)=mf(B^n,k)$ for $k=1,2$, together with the inequality $f(mB^n,3)\le mf(B^n,3)$, suggest that $f(mB^n,k)$ may satisfy a subadditivity property in the parameter $m$.
\begin{con}[Subadditivity]
	For $n\ge 2$, $k\ge 3$, and positive integers $m_1,m_2$,
	\[
	f\bigl((m_1+m_2)B^n,k\bigr)\le f(m_1B^n,k)+f(m_2B^n,k).
	\]
\end{con}
\begin{remark}
Simply combining two translated covers fails on points with coordinates in both subcubes. For example, when $n=2$ and $k=3$, the five hyperplanes
\[
x_1=1,\quad x_2=1,\quad
x_1+x_2=1\ \text{(twice)},\quad x_1+x_2=2
\]
cover $B^2\setminus\{\bm{0}\}$ three times, but their union with its translate by $(1,1)$ covers $(2,0)$ only twice. Moreover, because $U_{m,n,k}$ is linear in $m$, the present construction gives only
\[
f((m_1+m_2)B^n,k)\le U_{m_1,n,k}+U_{m_2,n,k},
\]
which does not imply the conjecture because $f(m_iB^n,k)$ may be strictly smaller than $U_{m_i,n,k}$.
\end{remark}
Theorem \ref{main-1} gives $2n+5\le f(2B^n,3)\le 2n+6$. By computer-aided verification (see Appendix \ref{appendix} for details), we obtain $f(2B^n,3)=2n+5$ for $n=2,3,4$. Therefore, we propose the following conjecture.
\begin{con}  
	For $n\ge 5$, $f(2B^n,3)=2n+5$.
\end{con}
Most of this paper, including Sections \ref{Sec3}--\ref{Sec7}, is devoted to the proof of the following theorem, an extension of the Sauermann--Wigderson Combinatorial Nullstellensatz from $B^n$
to $mB^n$. During the proof, we observe that the dimensional hypothesis in the
second case can be improved from $n\ge2k-3$ to $n\ge k-1$.
\begin{Theorem}\label{main-2}
    For $k\ge 2$ and $0\le l\le k-1$, let $P\in\mathbb{R}[x_1,\ldots,x_n]$ be a polynomial that has zeros of multiplicity at least $k$ at all points in $mB^n \setminus \{\bm{0}\}$, and a zero of multiplicity exactly $l$ at $\bm{0}$. If $l=k-1$, then
    \begin{equation*}
    \deg P\ge mn+(m+1)(k-1). \tag{3} \label{eq3}
     \end{equation*}
    If $0\le l\le k-2$ and $n\ge k-1$, then
    \begin{equation*}
    \deg P \ge mn+(m+1)(k-1)-1. \tag{4}\label{eq4}
     \end{equation*}
    In both cases the bound is attained for every $l$ in the stated range.
\end{Theorem}
\begin{coro}\label{coro1.8}
	For $k\ge 2$, the minimum size of a multiset of hyperplanes that covers every point of $mB^n\setminus \{\bm{0}\}$ at least $k$ times and covers $\bm{0}$ exactly $k-1$ times is
	\[
	mn+(m+1)(k-1).
	\]
\end{coro}

We conclude the introduction with the organization and logical structure of
the proofs. Section \ref{Sec2} proves Theorem \ref{main-1}: the upper bounds
come from explicit hyperplane constructions, while the lower bound in
\eqref{eq2} follows by associating a product polynomial with a covering and
applying Theorem \ref{main-2}. Section \ref{Sec3} proves Theorem
\ref{main-2}: it derives \eqref{eq3} using the reduction to
$(m,k)$-reduced polynomials and reduces \eqref{eq4} to the injectivity
of the highest-degree component map in Lemma \ref{c11}.
Section
\ref{Sec4} proves Lemma \ref{c11} using Lemma \ref{p1}. Section
\ref{Sec5} constructs a symmetric reduced polynomial and reduces
Lemma \ref{p1} to Claims \ref{l33} and \ref{l34}.
Sections \ref{Sec6} and \ref{Sec7} prove these two claims, respectively; the
coefficient evaluation in Section \ref{Sec7} uses the Lagrange inversion
formula. Appendix
\ref{appendix} gives the explicit constructions for the stated small
cases of $f(2B^n,3)$.

\section{Covering Bounds and Constructions} \label{Sec2}
 This section is devoted to proving Theorem \ref{main-1}, whose proof relies on Ball and Serra's Punctured Combinatorial Nullstellensatz \cite[Theorem 4.1]{Ball}; see also the subsequent erratum \cite{BallErratum}. We begin by recalling their result after introducing some necessary notation. Let $\mathbb{F}$ be a field and let $f$ be a nonzero polynomial in $\mathbb{F}[x_1,\ldots,x_n]$. We say that $\bm a=(a_1,\ldots,a_n)\in\mathbb{F}^n$ is a \emph{zero of multiplicity $k$} of $f$ if $k$ is the minimum degree of a monomial with nonzero coefficient in the expansion of $f(x_1+a_1,\ldots,x_n+a_n)$. For $i = 1,\ldots,n$, let $D_i$ and $S_i$ be finite nonempty subsets of $\mathbb{F}$ such that $D_i \subseteq S_i$, and define
    \[
    g_i(x_1,\ldots,x_n) = \prod_{s \in S_i} (x_i - s) \quad \text{and} \quad h_i(x_1,\ldots,x_n) = \prod_{d \in D_i} (x_i - d).
    \]
    For $k\ge1$, let $T(n,k)$ be the set of all nondecreasing $k$-tuples with entries in $\{1,\ldots,n\}$. For $\tau\in T(n,k)$, write $\tau(r)$ for its $r$-th entry.

\begin{Lemma}[Ball and Serra, 2009]\label{ba}
    If $f$ has a zero of multiplicity at least $k$ at all common zeros of $g_1, g_2, \ldots, g_n$, except for at least one point in $D_1 \times D_2 \times \cdots \times D_n$ where it has a zero of multiplicity less than $k$, then there exist polynomials $u_{\tau}$ for all $\tau\in T(n,k)$ satisfying $\deg(u_{\tau}) \le \deg(f) - \sum_{r=1}^{k} \deg(g_{\tau(r)})$ and a nonzero polynomial $v$ satisfying $\deg(v) \le \deg(f) - \sum_{i=1}^{n} \bigl(\deg(g_i) - \deg(h_i)\bigr)$ such that
        \[
        f = \sum_{\tau \in T(n,k)} g_{\tau(1)} \cdots g_{\tau(k)} u_{\tau} + v \prod_{i=1}^{n} \frac{g_i}{h_i}.
        \]
        Moreover, if there exists a point in $D_1 \times D_2 \times \cdots \times D_n$ where $f$ is nonzero, then
        \[
        \deg(f) \ge (k-1) \max_{j} (|S_j| - |D_j|) + \sum_{i=1}^{n} (|S_i| - |D_i|).
        \]
\end{Lemma}
    For any point $\bm{a}\in mB^n$, denote by $f(mB^n,k,\bm{a})$ the minimum size of a multiset of hyperplanes that covers each point of $mB^n \setminus \{\bm a\}$ at least $k$ times with $\bm a$ left uncovered. Clearly, $f(mB^n,k,\bm0)=f(mB^n,k)$. Below we show that the identities in \eqref{eq1} of Theorem \ref{main-1} hold not only for $f(mB^n,k)$, but also for $f(mB^n,k,\bm a)$ when $k=1,2$.

    Let $\mathcal H$ be any such covering multiset. Choose a defining linear polynomial $L_H$ for each $H\in\mathcal H$, and set
    $P=\prod_{H\in\mathcal H}L_H$. Then $\deg P=|\mathcal H|$. Every
    point covered exactly $r$ times by $\mathcal H$ is a zero of
    multiplicity $r$ of $P$. Hence $P(\bm a)\ne0$, and $P$ has zeros of
    multiplicity at least $k$ at all points in
    $mB^n\setminus\{\bm a\}$. Applying Lemma \ref{ba} with $S_i=\{0,1,\ldots,m\}$ and $D_i=\{a_i\}$ gives
    \[f(mB^n,k,\bm{a}) \ge mn + (k-1)m.\] 
    Thus, for $k=1,2$, the equality $f(mB^n,k,\bm{a})= mn + (k-1)m$ holds if and only if there exists a family of $mn + (k-1)m$ hyperplanes that cover each point of $mB^n \setminus \{\bm a\}$ at least $k$ times with $\bm{a}$ left uncovered.
\begin{Theorem}\label{2-cover}
    Let $n\ge 2$. For any $\bm{a} \in mB^n$,
    \begin{equation*}
       f(mB^n,1,\bm{a})=mn\quad \text{and} \quad  f(mB^n,2,\bm{a}) = mn + m.
    \end{equation*}
\end{Theorem}
\begin{proof}
    Let $\bm{a} = (a_1, \ldots, a_n)$. Consider the $mn$ hyperplanes defined by $x_i=j$, where $i\in\{1,\ldots,n\}$ and $j\in\{0,1,\ldots,m\}\setminus\{a_i\}$. A simple verification shows that $\bm a$ is uncovered, points differing from $\bm a$ in exactly one coordinate are covered once, and all remaining points are covered at least twice.
    Hence $f(mB^n,1,\bm a)=mn$. To prove that $f(mB^n,2,\bm a)=mn+m$, it suffices to construct $m$ hyperplanes that cover all points that differ from $\bm{a}$ in exactly one coordinate, but do not pass through $\bm{a}$. For $i=1,\ldots,n$, enumerate the elements of $\{0,1,\ldots,m\}\setminus\{a_i\}$ as $b_i^{(1)},\ldots,b_i^{(m)}$. For $s=1,\ldots,m$, define
    \[
    H_s:\quad
    \sum_{i=1}^n\frac{x_i-a_i}{b_i^{(s)}-a_i}=1.
    \]
    Each $H_s$ avoids $\bm a$ and passes through the point
    $(a_1,\ldots,a_{i-1},b_i^{(s)},a_{i+1},\ldots,a_n)$ for every $i$.
    Therefore, $H_1,\ldots,H_m$ cover every point differing from
    $\bm a$ in exactly one coordinate.
\end{proof}

    To end this section, we turn to the bounds in Theorem \ref{main-1} \eqref{eq2}.
    We begin with the parameterized upper bound. Recall that
\[
U_{m,n,k}
=m\min_{1\le c\le k}
\left\{
cn+\rho_ck-\frac{c\rho_c(\rho_c+1)}2
\right\},
\]
where
$\rho_c=\min\left\{n,\left\lfloor(k-1)/c\right\rfloor\right\}$
for $c=1,\ldots,k$.

\begin{prop}\label{prop:parameterized-upper}
For $k\ge1$,
\[
f(mB^n,k)\le U_{m,n,k}.
\]
\end{prop}

\begin{proof}
    Fix $c\in\{1,\ldots,k\}$ and put $\rho=\rho_c$. Take $c$ copies
    of each coordinate hyperplane
    \[
    x_i=j,
    \qquad
    i=1,\ldots,n,\quad j=1,\ldots,m.
    \]
    In addition, for each $h=1,\ldots,\rho$, take $k-ch$ copies of each
    hyperplane
    \[
    x_1+\cdots+x_n=u,
    \qquad
    u=(h-1)m+1,\ldots,hm.
    \]
    Here $k-ch\ge1$ since $ch\le c\rho\le k-1$. All these hyperplanes
    avoid $\bm{0}$, and their total number is
    \[
    cmn+m\sum_{h=1}^{\rho}(k-ch)
    =m\left(cn+\rho k-\frac{c\rho(\rho+1)}{2}\right).
    \]
    To verify the covering multiplicity, fix
    $\bm b=(b_1,\ldots,b_n)\in mB^n\setminus\{\bm{0}\}$. Let $q$ be
    the number of nonzero coordinates of $\bm b$, and set
    \[
    u=b_1+\cdots+b_n,
    \qquad
    h=\left\lceil\frac{u}{m}\right\rceil.
    \]
    Since $1\le u\le qm$, we have $1\le h\le q$. The coordinate
    hyperplanes cover $\bm b$ exactly $cq$ times. If $h\le\rho$, then
    $(h-1)m<u\le hm$, so the construction contains $k-ch$ copies of
    the hyperplane $x_1+\cdots+x_n=u$ through $\bm b$. Hence $\bm b$
    is covered
    \[
    cq+k-ch=k+c(q-h)\ge k
    \]
    times. If $h>\rho$, then $\rho<n$ because $h\le q\le n$. It follows that
    $\rho=\lfloor(k-1)/c\rfloor$, and hence $c(\rho+1)\ge k$. Since
    $q\ge h\ge\rho+1$, we have $cq\ge k$. Thus the coordinate hyperplanes
    alone cover $\bm b$ at least $k$ times. Taking the minimum over
    $c\in\{1,\ldots,k\}$ gives $f(mB^n,k)\le U_{m,n,k}$.
\end{proof}

When $n<k-1$, taking $c=1$ gives
$U_{m,n,k}\le
mn+m\binom{k}{2}-m\binom{k-n}{2}
<mn+m\binom{k}{2}$.
In the range $k\ge3$ and $n\ge k-1$ of Theorem
\ref{main-1} \eqref{eq2}, the following proposition determines
exactly when the choice $c=1$ is optimal.

\begin{prop}\label{prop:c-one-threshold}
Let $k\ge3$ and $n\ge k-1$. In the definition of $U_{m,n,k}$, the choice
$c=1$ is optimal if and only if
\[
n\ge \left\lfloor\frac{k^2}{4}\right\rfloor.
\]
It is the unique optimal choice when the inequality is strict,
whereas at $n=\lfloor k^2/4\rfloor$ the only optimal choices are
$c=1$ and $c=2$.
\end{prop}

\begin{proof}
Since $n\ge k-1$, we have
$\rho_c=\lfloor(k-1)/c\rfloor$ for $c=1,\ldots,k$. Define
\[
G_c(n)
=cn+\rho_ck-\frac{c\rho_c(\rho_c+1)}{2}.
\]
Then $U_{m,n,k}=m\min_{1\le c\le k}G_c(n)$ and
$G_1(n)=n+\binom{k}{2}$. Put $T=\lfloor k^2/4\rfloor$. For $c=2$,
direct substitution gives
\[
G_2(n)-G_1(n)=n-T.
\]
Now let $3\le c\le k$ and suppose that $n\ge T$. We have
\[
G_c(n)-G_1(n)
=(c-1)n-\binom{k}{2}
+\rho_c\left(k-\frac{c(\rho_c+1)}{2}\right).
\]
Since $c\ge3$ and $n\ge T$,
\[
(c-1)n-\binom{k}{2}
\ge2T-\binom{k}{2}
=\left\lfloor\frac{k}{2}\right\rfloor
>0.
\]
Moreover, since $c\rho_c\le k-1$ and $c\le k$,
\[
k-\frac{c(\rho_c+1)}{2}
=
\frac{(k-c\rho_c)+(k-c)}{2}
\ge\frac12.
\]
Since $\rho_c\ge0$, this implies
$\rho_c\bigl(k-c(\rho_c+1)/2\bigr)\ge0$.
Therefore, $G_c(n)>G_1(n)$ for every $3\le c\le k$ whenever
$n\ge T$. If $n<T$, then $G_2(n)<G_1(n)$, so $c=1$ is not optimal.
If $n=T$, then $G_2(n)=G_1(n)$, while
$G_c(n)>G_1(n)$ for every $3\le c\le k$. Thus $c=1$ and $c=2$
are the only optimal choices. If $n>T$, then
$G_2(n)>G_1(n)$ as well, so $c=1$ is the unique optimal choice.
\end{proof}

\begin{proof}[\bf{Proof of Theorem \ref{main-1} \eqref{eq2}}]
    Let $\mathcal H$ be a covering multiset of size $f(mB^n,k)$, and let
    $P$ be the product of its defining linear polynomials. Then
    $\deg P=|\mathcal H|=f(mB^n,k)$, $P(\bm 0)\ne0$, and
    $P$ has zeros of multiplicity at least $k$ at all points in
    $mB^n\setminus\{\bm 0\}$. Applying Theorem \ref{main-2} \eqref{eq4}
    with $l=0$ gives the lower bound in \eqref{eq2}. The upper bound
    follows from Proposition \ref{prop:parameterized-upper}.
\end{proof}

\section{Reduced Polynomials and Degree Bounds}\label{Sec3}
To prove Theorem \ref{main-2}, we establish several supporting lemmas. The proof of Lemma \ref{c11} is deferred to the next section. For $k\ge 2$, a polynomial $P\in\mathbb{R}[x_1,\ldots,x_n]$ is called \emph{$(m,k)$-reduced} if $\deg P \le mn+(m+1)(k-1)-1$ and no monomial of $P$ is divisible by $x_{i_1}^{m+1}\cdots x_{i_k}^{m+1}$ for any indices $i_1,\ldots,i_k$ (not necessarily distinct). Let $V_{m,k}$ denote the set of all $(m,k)$-reduced polynomials. It
is easily seen that $V_{m,k}$ forms a subspace of
$\mathbb{R}[x_1,\ldots,x_n]$. For each variable $x_i$, define the falling and rising factorials by $(x_i)^{\underline{0}}=(x_i)^{\overline{0}}=1$ and, for $r\ge 1$,
\[
(x_i)^{\underline{r}}=x_i(x_i-1)\cdots(x_i-r+1),\qquad
(x_i)^{\overline{r}}=x_i(x_i+1)\cdots(x_i+r-1).
\]

\begin{Lemma}\label{c1}
	Let $k\ge 2$. For every polynomial $P\in \mathbb{R}[x_1,\ldots,x_n]$, there exists a polynomial $P_0\in V_{m,k}$ with $\deg P_0 \le \deg P$ such that $P-P_0$ has zeros of multiplicity at least $k$ at
	all points in $mB^n \setminus \{\bm{0}\}$ and a zero of multiplicity at least $k-1$ at $\bm{0}$. Furthermore, if $\deg P\le mn+(m+1)(k-1)-1$, then one may choose $P_0\in V_{m,k}$ with $\deg P_0 \le \deg P$ such that $P-P_0$ has zeros of multiplicity at least $k$ at all points in $mB^n$.
\end{Lemma}

\begin{proof}
Let $V_1$ be the ideal of $\mathbb{R}[x_1,\ldots,x_n]$ generated by the polynomials
	\[
\prod_{j=1}^{k} (x_{i_j})^{\underline{m+1}} \quad \text{with}\;  i_1,\ldots,i_k \in \{1,\ldots,n\}
	\]
and let $V_2$ be the $\mathbb{R}$-linear space spanned by the polynomials
	\[
\prod_{i=1}^{n}(x_{i} - 1)^{\underline{m}} \cdot \prod_{j=1}^{k-1}(x_{i_j})^{\underline{m+1}} \quad \text{with}\;   i_1,\ldots,i_{k-1} \in \{1,\ldots,n\}.
	\]
Then every element of $V_1$ has zeros of multiplicity at least $k$ at all points in $mB^n$, and every element of $V_2$ has zeros of multiplicity at least $k$ at all points in $mB^n \setminus \{\bm{0}\}$ and a zero of multiplicity at least $k-1$ at $\bm{0}$.  To prove the existence of $P_0$, it suffices to show that
	\begin{equation*}
		\begin{aligned}
			\mathbb{R}[x_1,\ldots,x_n] &= V_{m,k}
			+ V_1  + V_2.
		\end{aligned}
	\end{equation*}
Note that the polynomials 
\begin{equation}\label{2polynomials}
x_1^{l_1}\cdots x_n^{l_n} \cdot \prod_{j=1}^{k} (x_{i_j})^{\underline{m+1}} \quad \text{and} \quad \prod_{i=1}^{n}(x_{i} - 1)^{\underline{m}} \cdot \prod_{j=1}^{k-1}(x_{i_j})^{\underline{m+1}}
\end{equation}
belong to $V_1$ and $V_2$ respectively. Their respective leading monomials are
\begin{equation}\label{2monomials}
	x_{i_1}^{m+1}\cdots x_{i_k}^{m+1} \cdot x_1^{l_1}\cdots x_n^{l_n} \quad \text{and} \quad x_{i_1}^{m+1}\cdots x_{i_{k-1}}^{m+1} \cdot (x_1\cdots x_n)^{m}.
\end{equation}
Now suppose that $P$ contains a monomial of one of the forms in
\eqref{2monomials}. Choose such a monomial of maximal degree and
subtract an appropriate scalar multiple of the corresponding
polynomial in \eqref{2polynomials}. This cancels the chosen monomial
and introduces only monomials of strictly smaller degree. By repeating this subtraction finitely many times, we can obtain a polynomial $P_0$ such that $P_0$ contains no monomials of the forms in \eqref{2monomials} and
\[
P-P_0\in V_1+V_2.
\] 
It remains to show that $P_0\in V_{m,k}$. Since $P_0$ contains no monomial of the form $x_{i_1}^{m+1}\cdots x_{i_k}^{m+1} \cdot x_1^{l_1}\cdots x_n^{l_n}$, we observe the following two facts. First, no monomial of $P_0$ is divisible by $x_{i_1}^{m+1}\cdots x_{i_k}^{m+1}$. Second, among all possible monomials in $P_0$, any monomial of the maximal degree must be of the form $x_{i_1}^{m+1}\cdots x_{i_{k-1}}^{m+1} \cdot (x_1\cdots x_n)^{m}$. Consequently, we have the inequality $\deg(P_0)\le mn+(m+1)(k-1)$. Moreover, because $P_0$ also contains no monomial of the form $x_{i_1}^{m+1}\cdots x_{i_{k-1}}^{m+1} \cdot (x_1\cdots x_n)^{m}$, the inequality is strict, i.e.,
\[
\deg(P_0)\le mn+(m+1)(k-1)-1.
\]
So far we have proved $P_0 \in V_{m,k}$. Now consider the case $\deg P\le mn+(m+1)(k-1)-1$. In this case, $P$ contains no monomial of the form $x_{i_1}^{m+1}\cdots x_{i_{k-1}}^{m+1} \cdot (x_1\cdots x_n)^{m}$. The above arguments imply that $P_0$ is obtained from $P$ by subtracting polynomials of the form $x_1^{l_1}\cdots x_n^{l_n} \cdot \prod_{j=1}^{k} (x_{i_j})^{\underline{m+1}}$. It follows that $P-P_0\in V_1$, and hence $P-P_0$ has zeros of multiplicity at least $k$ at all points in $mB^n$.
\end{proof}

For $k\ge1$, let $M_k(n)=\binom{n+k-1}{n}$ be the number of
$n$-tuples $(l_1,\ldots,l_n)$ of nonnegative integers with
$l_1+\cdots+l_n<k$. Equivalently, $M_k(n)$ is the number of derivative
operators of order less than $k$ with respect to $x_1,\ldots,x_n$.
For $k\ge2$, set
$N=\bigl((m+1)^n-1\bigr)M_k(n)+M_{k-1}(n)$, and define a linear map
\[
\psi_{m,k}:V_{m,k}\rightarrow \mathbb{R}^N
\]
by sending each $P\in V_{m,k}$ to the $N$-tuple consisting of all
derivatives of $P$ of order less than $k$ at all points in
$mB^n \setminus \{\bm{0}\}$, and all derivatives of $P$ of order
less than $k-1$ at $\bm{0}$.

\begin{Lemma}\label{c2}
	For $k\ge 2$, we have $\dim V_{m,k}=N$.
\end{Lemma}
\begin{proof}
	By definition, a basis of $V_{m,k}$ consists of monomials of the form 
	\[x_1^{(m+1)l_1+r_1} \cdots x_n^{(m+1)l_n+r_n}\]
	with nonnegative integers $l_1,\ldots,l_n$ satisfying $l_1+\cdots+l_n<k$ and $r_i\in \{0,1,\ldots ,m\}$ such that 
	\[
	(m+1)(l_1+\cdots+l_n)+r_1+\cdots +r_n\le mn+(m+1)(k-1)-1.
	\]
	If $r_1+\cdots+r_n=mn$, then we have $(m+1)(l_1+\cdots+l_n)\le (m+1)(k-1)-1$, which is equivalent to $l_1+\cdots+l_n<k-1$. Hence, there are $M_{k-1}(n)$ choices for $\left(l_1,\ldots,l_n\right)$.
	If $r_1+\cdots +r_n\le mn-1$, there are $(m+1)^n-1$ choices for $(r_1,\ldots,r_n)$, and for each such choice, the degree condition reduces to $l_1+\cdots+l_n<k$, yielding $M_k(n)$ choices for $(l_1,\ldots,l_n)$.
	Thus, $\dim V_{m,k}=\bigl((m+1)^n-1\bigr)M_k(n)+M_{k-1}(n)=N$.
\end{proof}
By combining Lemmas \ref{c1} and \ref{c2}, we obtain the following result.

\begin{Lemma}\label{c3} 
	For $k\ge 2$, the map $\psi_{m,k}:V_{m,k}\rightarrow \mathbb{R}^N$ is an isomorphism.
\end{Lemma}
\begin{proof}
	Since $\dim V_{m,k} = N$ by Lemma \ref{c2}, it suffices to prove that $\psi_{m,k} \colon V_{m,k} \to \mathbb{R}^N$ is surjective. For any $N$-tuple $\alpha \in \mathbb{R}^N$, we can construct a polynomial $P$ of sufficiently large degree whose derivatives of order less than $k$ at all points in $mB^n \setminus \{\bm{0}\}$, and of order less than $k-1$ at $\bm{0}$, coincide with the entries of $\alpha$.
	By Lemma \ref{c1}, there exists $P_0 \in V_{m,k}$ such that $P-P_0$ has zeros of multiplicity at least $k$ at all points in $mB^n\setminus\{\bm{0}\}$ and a zero of multiplicity at least $k-1$ at $\bm{0}$. Consequently, all derivative values of $P-P_0$ used in the definition
	of $\psi_{m,k}$ are zero, so $P_0$ has the same prescribed derivative
	data as $P$:
	\[
	\psi_{m,k}(P_0) = \alpha.
	\]	
	It remains to construct $P$. For a multi-index
	$\bm{i}=(i_1,\ldots,i_n)$, write
	$|\bm{i}|=i_1+\cdots+i_n$ and
	$\partial^{\bm{i}}
	=\partial^{|\bm{i}|}/(\partial x_1^{i_1}\cdots\partial x_n^{i_n})$.
	For each multi-index $\bm i$ with $|\bm i|\le k-1$ and each
	$\bm b=(b_1,\ldots,b_n)\in mB^n$, define
	\begin{equation*}
		P_{\bm{i},\bm{b}} = (x_1-b_1)^{i_1}\cdots (x_n-b_n)^{i_n} \prod_{t=1}^n \prod_{s \in \{0,1,\ldots,m\} \setminus \{b_t\}} (x_t - s)^k.
	\end{equation*}
	For every multi-index $\bm{j}$ with $|\bm j|\le k-1$, the following
	properties hold:
	\begin{itemize}
		\item[(1)] If $|\bm{j}| < |\bm{i}|$, then $\partial^{\bm j}P_{\bm{i},\bm{b}}$ vanishes on $mB^n$;
		\item[(2)] If $|\bm{j}| = |\bm{i}|$ and $\bm{j} \neq \bm{i}$, then $\partial^{\bm j}P_{\bm{i},\bm{b}}$ vanishes on $mB^n$;
		\item[(3)] If $\bm{j}=\bm{i}$, then $\partial^{\bm j}P_{\bm{i},\bm{b}}$ vanishes on $mB^n \setminus \{\bm{b}\}$ and is nonzero at $\bm{b}$.
	\end{itemize}
	Indeed, the value of the derivative in (3) at $\bm b$ is
	\[
	(\partial^{\bm i}P_{\bm i,\bm b})(\bm b)
	=
	i_1!\cdots i_n!
	\prod_{t=1}^n
	\prod_{s\in\{0,1,\ldots,m\}\setminus\{b_t\}}(b_t-s)^k\ne0.
	\]
	Order the multi-indices by nondecreasing total order and fix an arbitrary order on the points of $mB^n$. Use the induced lexicographic order on the pairs
	$(\bm i,\bm b)$ with $|\bm i|\le k-1$ and $\bm b\in mB^n$,
	excluding those with $\bm b=\bm{0}$ and $|\bm i|=k-1$, since
	$\psi_{m,k}$ records derivatives at $\bm{0}$ only up to order
	$k-2$. There are exactly $N$ such pairs. Let $A$ be the $N\times N$ matrix whose columns are indexed by $P_{\bm i,\bm b}$ and whose rows, in the same order, are indexed by
	\[
	F\longmapsto (\partial^{\bm j}F)(\bm c).
	\]
	Thus $A_{(\bm j,\bm c),(\bm i,\bm b)}=(\partial^{\bm j}P_{\bm i,\bm b})(\bm c)$. Properties (1)--(3) show that $A$ is lower triangular with nonzero diagonal entries. Hence $A\bm y=\alpha$ has a unique solution $\bm y=(y_{\bm i,\bm b})$, and the polynomial
	\begin{equation*}
		P = \sum_{(\bm i,\bm b)} y_{\bm{i},\bm{b}} P_{\bm{i},\bm{b}}
	\end{equation*}
	satisfies the required derivative conditions. This completes the proof.
\end{proof}

\begin{coro}\label{coro3.4}
	For each polynomial $P$, the polynomial $P_0$ in Lemma \ref{c1} is unique.
\end{coro}
\begin{proof}
	Suppose that $P_0,P_1\in V_{m,k}$ are two possible choices in Lemma \ref{c1} for the same polynomial $P$. Since both $P-P_0$ and $P-P_1$ have zeros of multiplicity at least
	$k$ at all points in $mB^n\setminus\{\bm{0}\}$ and a zero of
	multiplicity at least $k-1$ at $\bm{0}$, their difference
	$P_1-P_0$ satisfies the same multiplicity conditions. Consequently, all derivatives of $P_1-P_0$ of order less than $k$
	vanish at all points in $mB^n\setminus\{\bm{0}\}$, and all derivatives of
	order less than $k-1$ vanish at $\bm{0}$. Thus $\psi_{m,k}(P_1-P_0)=\bm{0}$, and Lemma \ref{c3} gives
	$P_1=P_0$.
\end{proof}

We now prove Theorem \ref{main-2} \eqref{eq3} using
Lemmas \ref{c1} and \ref{c3}.

\begin{proof}[\bf{Proof of Theorem \ref{main-2} \eqref{eq3}}]
    Suppose that there exists a polynomial $P\in \mathbb{R}[x_1,\ldots,x_n]$ of degree $\deg P \le mn+(m+1)(k-1)-1$ having zeros of multiplicity at least $k$ at all points in $mB^n \setminus \{\bm{0}\}$, and a zero of multiplicity exactly $k-1$ at $\bm{0}$. By Lemma \ref{c1}, there exists a polynomial $P_0\in V_{m,k}$ such that $P-P_0$ has zeros of multiplicity at least $k$ at all points in $mB^n$. Clearly, $P_0\ne0$. On the other hand, $P_0=P-(P-P_0)$ has zeros of multiplicity at
    least $k$ at all points in $mB^n\setminus\{\bm{0}\}$, and a zero of
    multiplicity exactly $k-1$ at $\bm{0}$. Hence, the derivatives of
    $P_0$ of order less than $k$ at all points in
    $mB^n\setminus\{\bm{0}\}$ are identically zero, and derivatives of
    $P_0$ of order less than $k-1$ at $\bm{0}$ are also zero. This
    means $\psi_{m,k}(P_0)=\bm{0}$. By Lemma \ref{c3}, we obtain
    $P_0=0$, a contradiction. Therefore, $\deg P \ge mn+(m+1)(k-1)$. The lower bound is attained by taking $P\left(x_1,\ldots,x_n\right)=x_1^{k-1}(x_1-1)^{k-1}\cdots(x_1-m)^{k-1}\cdot\prod_{i=1}^{n}(x_i-1)\cdots(x_i-m)$.
\end{proof}

\begin{proof}[\bf{Proof of Corollary \ref{coro1.8}}]
The lower bound follows from Theorem \ref{main-2} \eqref{eq3}, and
equality is attained by the construction in its proof.
\end{proof}

To prove Theorem \ref{main-2} \eqref{eq4}, a key step is to establish Lemma \ref{c11}. For $k\ge 2$, let $U_{m,k}\subseteq V_{m,k}$ denote the subspace of all $(m,k)$-reduced polynomials which have zeros of multiplicity at least $k$ at all points in $mB^n \setminus \{\bm{0}\}$. Let $W_{m,k}\subseteq \mathbb{R}[x_1,\ldots,x_n]$ denote the subspace spanned by all polynomials
\begin{equation}\label{formula1}
    (x_1\cdots x_n)^m (x_1^{l_1} \cdots x_n^{l_n})^{m+1} (x_1^d+\cdots+x_n^d)
\end{equation}
for nonnegative integers $d,l_1,\ldots,l_n$ satisfying
\[
d+(m+1)(l_1+\cdots+l_n)=(m+1)(k-1)-1. 
\]
Define a linear map 
\[
\varphi_{m,k} :U_{m,k}\rightarrow \mathbb{R}[x_1,\ldots,x_n]
\]
by sending each $P\in U_{m,k}$ to its homogeneous component of degree $mn+(m+1)(k-1)-1$.
\begin{Lemma}\label{c11}
For $k\ge 2$ and $n\ge k-1$, the map $\varphi_{m,k}$ is injective and
 \[\varphi_{m,k} (U_{m,k})=W_{m,k}.\]
\end{Lemma}
The proof of Lemma \ref{c11} is given in Section \ref{Sec4} and relies
on auxiliary results established in Sections \ref{Sec5}--\ref{Sec7}.
We next use the lemma to prove Theorem \ref{main-2} \eqref{eq4}.
\begin{proof}[\bf{Proof of Theorem \ref{main-2} \eqref{eq4}}]
Let $k\ge 2$ and $n\ge k-1$.  Let $P\in\mathbb{R}[x_1,\ldots,x_n]$ be a polynomial having zeros of multiplicity at least $k$ at all points in $mB^n \setminus \{\bm{0}\}$, and a zero of multiplicity at most $k-2$ at $\bm{0}$. By Lemma \ref{c1}, there exists a polynomial $P_0\in V_{m,k}$ with $\deg P_0\le \deg P$ such that $P-P_0$ has zeros of multiplicity at least $k$ at
    all points in $mB^n \setminus \{\bm{0}\}$ and a zero of multiplicity at least $k-1$ at $\bm{0}$. Since both $P$ and $P-P_0$ have zeros of multiplicity at least $k$ at all points in $mB^n \setminus \{\bm{0}\}$, so does $P_0$. Hence $P_0\in U_{m,k}$.
    Clearly, $P_0\ne0$.
    By Lemma \ref{c11}, we have $\varphi_{m,k}(P_0)\ne 0$, i.e., the homogeneous component of $P_0$ of degree $mn+(m+1)(k-1)-1$ is nonzero. Thus, $\deg P\ge \deg P_0\ge  mn+(m+1)(k-1)-1$.
It remains to show that this lower bound is attained.  For any $l\in\{0,1,\ldots,k-2\}$, consider the polynomial
    \begin{equation*}
        P^{*}=x_1^{l}\cdot \prod_{i=1}^{n}(x_{i} - 1)^k \cdots (x_{i} - m)^k.
    \end{equation*}
Then $P^{*}$ has zeros of multiplicity at least $k$ at all points in $mB^n \setminus \{\bm{0}\}$ and a zero of multiplicity exactly $l$ at $\bm{0}$. By Lemma \ref{c1}, there exists $P_0^*\in V_{m,k}$ satisfying the same multiplicity conditions as $P^*$. Consequently, $\deg P_0^*\ge mn+(m+1)(k-1)-1$. Note that $P_0^*\in V_{m,k}$. The definition of $(m,k)$-reduced polynomials implies the reverse inequality: $\deg P_0^*\le mn+(m+1)(k-1)-1$. Thus $\deg P_0^*=mn+(m+1)(k-1)-1$.
\end{proof}

\section{The Highest-Degree Component Map}\label{Sec4}
This section proves Lemma \ref{c11} using Lemma \ref{p1}, whose proof is given in the next section.
For $k\ge 2$, recall that $U_{m,k}\subseteq V_{m,k}$ denotes the subspace of all $(m,k)$-reduced polynomials which have zeros of multiplicity at least $k$ at all points in $mB^n \setminus \{\bm{0}\}$. Before proceeding to the proof of Lemma \ref{c11}, we establish several preliminary lemmas.

\begin{Lemma}\label{fact1}
	For $k\ge 2$, we have $\dim U_{m,k}=M_{k-1}(n)$.
\end{Lemma}
\begin{proof} 
	Recall that $N=\bigl((m+1)^n-1\bigr)M_k(n)+M_{k-1}(n)$. 
	By definition, $U_{m,k}$ consists exactly of those polynomials in $V_{m,k}$ whose image under the map $\psi_{m,k}:V_{m,k}\rightarrow \mathbb{R}^N$ has its first $\left(\left(m+1\right)^n-1\right)\cdot M_k(n)$ entries equal to zero. The subspace consisting of such $N$-tuples has dimension  $M_{k-1}(n)$. Since $\psi_{m,k}$ is an isomorphism by Lemma \ref{c3}, we have $\dim U_{m,k}=M_{k-1}(n)$.
\end{proof}

\begin{Lemma}\label{fact2}
    Let $ k\ge 3$. Then for each $j\in\left\{1,\ldots,n\right\}$ and every polynomial $P\in U_{m,k-1}$, we have $x_j(x_j-1)\cdots(x_j-m)\cdot P\in U_{m,k}$.
\end{Lemma}
\begin{proof}
    Since $P\in U_{m,k-1}\subseteq V_{m,k-1}$ is $(m,k-1)$-reduced, we have $\deg P\le mn+(m+1)(k-2)-1$, and no monomial of $P$ is divisible by $x_{i_1}^{m+1}\cdots x_{i_{k-1}}^{m+1}$. It follows that $x_j(x_j-1)\cdots(x_j-m)\cdot P$ has degree at most $mn+(m+1)(k-1)-1$ and contains no monomial divisible by $x_{i_1}^{m+1}\cdots x_{i_{k}}^{m+1}$ (otherwise $P$ would contain a monomial divisible by $x_{i_1}^{m+1}\cdots x_{i_{k-1}}^{m+1}$). Hence $x_j(x_j-1)\cdots(x_j-m)\cdot P\in V_{m,k}$.
    Moreover, since $P\in U_{m,k-1}$ has zeros of multiplicity at least $k-1$ at all points in $mB^n \setminus \{\bm{0}\}$, the polynomial $x_j(x_j-1)\cdots(x_j-m)\cdot P$ has zeros of multiplicity at least $k$ at these points. Therefore  $x_j(x_j-1)\cdots(x_j-m)\cdot P\in U_{m,k}$.
\end{proof}

Recall that $W_{m,k}\subseteq \mathbb{R}[x_1,\ldots,x_n]$ is spanned by all polynomials
\[
    (x_1\cdots x_n)^m (x_1^{l_1} \cdots x_n^{l_n})^{m+1} (x_1^d+\cdots+x_n^d),
\]
where $d, l_1, \ldots, l_n$ are nonnegative integers satisfying
\[
d + (m+1)(l_1+\cdots+l_n) = (m+1)(k-1)-1,
\]
as in \eqref{formula1}. Every polynomial in $W_{m,k}$ is homogeneous of degree $mn+(m+1)(k-1)-1$. 
\begin{Lemma}\label{fact3}
	For $k\ge 2$ and $n\ge k-1$, the polynomials in \eqref{formula1} form a basis of $W_{m,k}$.
\end{Lemma}
\begin{proof}
	Since $W_{m,k}$ is spanned by all polynomials in \eqref{formula1}, it suffices to show that
	these polynomials are linearly independent. Dividing each polynomial by $(x_1\cdots x_n)^m$, it remains to prove the linear independence of the polynomials
	\begin{equation*}
		Q_{d,l_1,\ldots,l_n}= (x_1^{l_1} \cdots x_n^{l_n})^{m+1} (x_1^d+\cdots+x_n^d),
	\end{equation*}
	where $d,l_1,\ldots,l_n$ are nonnegative integers satisfying $d+(m+1)(l_1+\cdots+l_n)=(m+1)(k-1)-1$. In particular, $d \equiv m \pmod{m+1}$.
	Hence, every monomial of $Q_{d,l_1,\ldots,l_n}$ contains exactly one variable whose exponent is congruent to $m$ modulo $m+1$, and this exponent is at least $d$. 
	Note that $l_1+\cdots+l_n\le k-2 <n$, so there
	exists an index $i$ such that $l_i=0$. For such $i$, the monomial
	$x_i^d\cdot (x_1^{l_1} \cdots x_{i-1}^{l_{i-1}}   x_{i+1}^{l_{i+1}} \cdots x_n^{l_n})^{m+1}$
	appears in $Q_{d,l_1,\ldots,l_n}$, and the exponent of $x_i$ in this monomial is exactly $d$.
	Suppose that there is a nontrivial linear relation
	\begin{equation*}
		\sum_{\left(d,l_1,\ldots,l_n\right)}\lambda_{d,l_1,\ldots,l_n}Q_{d,l_1,\ldots,l_n}=0
	\end{equation*}
	with coefficients $\lambda_{d,l_1,\ldots,l_n}\in \mathbb{R}$ not all zero. Let $d^{*}$ be the minimum value of $d$ for which some coefficient $\lambda_{d,l_1,\ldots,l_n}$ is nonzero. Note that $d^{*} \equiv m \pmod{m+1}$.
	Fix an index $i$ such that $\lambda_{d,l_1,\ldots,l_n}\ne 0$ for some $\left(d,l_1,\ldots,l_n\right)$ with $d=d^{*}$ and $l_i=0$.
	Consider all monomials of polynomials $Q_{d,l_1,\ldots,l_n}$ with $\lambda_{d,l_1,\ldots,l_n} \ne 0$ for which $x_i$ has exponent $d^{*}$. Such monomials can only arise from those polynomials
	$Q_{d,l_1,\ldots,l_n}$ with $d=d^*$ and $l_i=0$. Moreover, each such polynomial contributes exactly one monomial in which $x_i$ has exponent $d^*$, namely
	\[
	x_i^{d^*} (x_1^{l_1}\cdots x_{i-1}^{l_{i-1}}   x_{i+1}^{l_{i+1}} \cdots x_n^{l_n} )^{m+1} .
	\]
	Distinct tuples $(l_1,\ldots,l_n)$ with $l_i=0$ yield distinct displayed monomials, so these monomials cannot cancel. This contradiction shows that all
	coefficients $\lambda_{d,l_1,\ldots,l_n}$ must vanish. Hence, the polynomials in \eqref{formula1} are linearly independent.
\end{proof}
\begin{Lemma}\label{c8}
	For $k\ge 2$ and $n\ge k-1$, we have $\dim W_{m,k}=\dim U_{m,k}$.
\end{Lemma}

\begin{proof}
	By Lemma \ref{fact3}, the dimension of $W_{m,k}$ equals the number of
	polynomials in \eqref{formula1}. Equivalently, it is the number of nonnegative integer tuples $\left(d,l_1,\ldots,l_n\right)$ satisfying
	\[
	d+(m+1)(l_1+\cdots+l_n)=(m+1)(k-1)-1.
	\] 
	This equality implies that $l_1+\cdots+l_n <k-1$. Hence, there are exactly $M_{k-1}(n)$ choices for $(l_1,\ldots,l_n)$. For each such choice, the integer $d$ is uniquely determined by $d=(m+1)(k-1)-1-(m+1)(l_1+\cdots+l_n)$. Therefore,
	\[
	\dim W_{m,k}=M_{k-1}(n)=\dim U_{m,k},
	\]
	where the last equality follows from Lemma \ref{fact1}.
\end{proof}

Recall that the map 
\[
\varphi_{m,k} :U_{m,k}\rightarrow \mathbb{R}[x_1,\ldots,x_n]
\]
sends each $P\in U_{m,k}$ to its homogeneous component of degree $mn+(m+1)(k-1)-1$.

\begin{Lemma}\label{p1}
    For $k\ge 2$ and $n\ge k-1$, there exists a polynomial $P\in U_{m,k}$ such that its homogeneous component $\varphi_{m,k} (P)$ of degree $mn+(m+1)(k-1)-1$ satisfies $\varphi_{m,k} (P)\in W_{m,k}$ and the following property: when expressing $\varphi_{m,k} (P)$ in terms of the basis of $W_{m,k}$ given in \eqref{formula1}, the coefficient of
    \begin{equation}\label{basisele}
    	(x_1\cdots x_n)^{m}   (x_1^{(m+1)(k-1)-1}+\cdots+x_n^{(m+1)(k-1)-1})
    \end{equation}
    is nonzero.
\end{Lemma}

The proof of Lemma \ref{p1} is postponed to Section \ref{Sec5}. Using this lemma, we establish Lemma \ref{c11}.

\begin{proof}[\bf{Proof of Lemma \ref{c11}}]
    By Lemma \ref{c8}, we have $\dim U_{m,k} = \dim W_{m,k}$. Hence, to prove that $\varphi_{m,k}$ is injective, it suffices to show that $\varphi_{m,k} (U_{m,k})=W_{m,k}$. Since $\dim \varphi_{m,k} (U_{m,k}) \le \dim U_{m,k} = \dim W_{m,k}$, the equality $\varphi_{m,k} (U_{m,k})=W_{m,k}$ is equivalent to the inclusion
    \[
    W_{m,k}\subseteq \varphi_{m,k} (U_{m,k})\quad \text{ for } k\ge 2,\text{ and } \; n\ge k-1.
    \]
    We prove this inclusion by induction on $k$.
    For $k=2$, note that the space 
    \begin{equation*}
        W_{m,2}=\text{span}_{\mathbb{R}}\left\{(x_1\cdots x_n)^{m} (x_1^m+\cdots+x_n^m)\right\}.
    \end{equation*}
    By Lemma \ref{p1} for $k=2$, there exist a polynomial $P\in U_{m,2}$ and a nonzero scalar $ c\in \mathbb{R}$ such that 
    \begin{equation*}
        \varphi_{m,2} (P)=c\cdot (x_1\cdots x_n)^{m}    (x_1^m+\cdots+x_n^m)\in W_{m,2}.
    \end{equation*}
    Hence $W_{m,2}\subseteq \varphi_{m,2} (U_{m,2})$.    
    Now suppose that $k\ge 3$ and $W_{m,k-1}\subseteq \varphi_{m,k-1} (U_{m,k-1})$. Let $W_{m,k}'\subseteq W_{m,k}$ be the subspace spanned by all polynomials of the form \eqref{formula1} with $\sum_{i=1}^{n}l_i \ge 1$. Equivalently, $W_{m,k}'$ is spanned by all basis elements of $W_{m,k}$ except \eqref{basisele}.
    Thus $\dim W_{m,k}'=\dim W_{m,k} -1$.
    We first show that $W_{m,k}'\subseteq \varphi_{m,k} (U_{m,k})$. Fix a basis element of $W_{m,k}'$, namely
    \[
    R=(x_1\cdots x_n)^m (x_1^{l_1} \cdots x_n^{l_n})^{m+1} (x_1^d+\cdots+x_n^d)
    \] 
    of the form \eqref{formula1} with $\sum_{i=1}^{n}l_i \ge 1$. Choose an index $j$ with $l_j\ge 1$, and define
    \begin{equation*}
        \widetilde R=\frac{R}{x_j^{m+1}}=(x_1\cdots x_n)^m (x_1^{l_1} \cdots x_{j}^{l_{j}-1}\cdots x_n^{l_n})^{m+1} (x_1^d+\cdots+x_n^d).
    \end{equation*}
    Then $\widetilde R \in W_{m,k-1} $. By the induction hypothesis, there exists a polynomial $P^{*}\in U_{m,k-1}$ of degree $mn+(m+1)(k-2)-1$ whose homogeneous component of that degree is 
    \[
    \varphi_{m,k-1} (P^{*})=\widetilde R.
    \]
    By Lemma \ref{fact2}, we have $x_j(x_j-1)\cdots(x_j-m)\cdot P^{*}\in U_{m,k}$, and its homogeneous component of degree $mn+(m+1)(k-1)-1$ is
    \begin{equation*}
        x_j^{m+1}\cdot \varphi_{m,k-1}(P^{*})=R.
    \end{equation*}
    Hence $R \in \varphi_{m,k}(U_{m,k})$.
    This implies that 
    $W_{m,k}'\subseteq \varphi_{m,k} (U_{m,k}).$
    Consequently, 
    \[
    W_{m,k}' \subseteq \varphi_{m,k}(U_{m,k}) \cap W_{m,k} \subseteq W_{m,k}.
    \]
    By Lemma \ref{p1}, there exists $P \in U_{m,k}$ such that $\varphi_{m,k}(P) \in W_{m,k}$ has a nonzero coefficient for the basis element \eqref{basisele} when expressed in terms of the basis of $W_{m,k}$. Recall that $W_{m,k}'$ is spanned by all basis elements of $W_{m,k}$ except \eqref{basisele}. Therefore 
    $\varphi_{m,k}(P) \in W_{m,k}  $ but $\varphi_{m,k}(P) \notin W_{m,k}'$.
    Thus $W_{m,k}' \subsetneq \varphi_{m,k}(U_{m,k}) \cap W_{m,k}$.
    Since $\dim W_{m,k}' = \dim W_{m,k} - 1$, it follows that $\varphi_{m,k}(U_{m,k}) \cap W_{m,k} = W_{m,k}$, i.e., $W_{m,k} \subseteq \varphi_{m,k}(U_{m,k})$. This completes the proof.
\end{proof}

\section{The Symmetric Polynomial Construction}\label{Sec5}
To prove Lemma \ref{p1}, we construct a symmetric polynomial
$P_0\in U_{m,k}$ with the required properties. Write $\bm x=(x_1,\ldots,x_n)$.
For positive integers $\lambda_1,\ldots,\lambda_t$ with $t\le n$,
define the augmented monomial symmetric polynomial and the
power-sum symmetric polynomial by
\begin{align*}
m_{\lambda_1,\ldots,\lambda_t}(\bm x)
&=
\sum_{\substack{i_1,\ldots,i_t\\ \text{distinct}}}
x_{i_1}^{\lambda_1}\cdots x_{i_t}^{\lambda_t},\\
p_{\lambda_1,\ldots,\lambda_t}(\bm x)
&=
\prod_{j=1}^t\left(\sum_{i=1}^n x_i^{\lambda_j}\right).
\end{align*}

\begin{Lemma}\label{c31}
	Let $k\ge 2$ and $n\ge k-1$. Suppose that $\lambda_1,\ldots,\lambda_t$ are positive integers with $\lambda_1+\cdots+\lambda_t \le (m+1)(k-1)-1$, such that all but one of the $\lambda_i$ are divisible by $m+1$, and the remaining one is congruent to $m$ modulo $m+1$. Then the augmented monomial symmetric polynomial
    $m_{\lambda_1,\ldots,\lambda_t}(\bm x)$
    can be written uniquely as a linear combination of power-sum symmetric polynomials
    $p_{\mu_1,\ldots,\mu_s}(\bm x)$,
	where $\mu_1,\ldots,\mu_s$ are positive integers with $\mu_1+\cdots+\mu_s=\lambda_1+\cdots+\lambda_t$, such that all but one of the $\mu_j$ are divisible by $m+1$, and the remaining one is congruent to $m$ modulo $m+1$.
\end{Lemma}
\begin{proof}
	We first prove existence by induction on $t$. For $t=1$, note that $\lambda_1 \equiv m \pmod{m+1}$, so $m_{\lambda_1}(\bm x)=p_{\lambda_1}(\bm x)$ already has the desired form.
	Suppose that the statement holds for $t-1$.
	Without loss of generality, we can assume that $\lambda_t$ is divisible by $m+1$. Note from the definitions that
	\begin{equation*}
			m_{\lambda_1,\ldots,\lambda_t}(\bm x)=p_{\lambda_t}(\bm x)m_{\lambda_1,\ldots,\lambda_{t-1}}(\bm x) -\sum_{j=1}^{t-1}m_{\lambda_1,\ldots,\lambda_j+\lambda_t,\ldots,\lambda_{t-1}}(\bm x).
	\end{equation*}
	Since exactly one of $\lambda_1,\ldots,\lambda_{t-1}$ is congruent to $m$ modulo $m+1$, we apply the induction hypothesis to
	$m_{\lambda_1,\ldots,\lambda_{t-1}}(\bm x)$.
	It follows that 
	$p_{\lambda_t}(\bm x)m_{\lambda_1,\ldots,\lambda_{t-1}}(\bm x)$
	is a linear combination of terms of the form 
	$p_{\lambda_t,\mu_1,\ldots,\mu_s}(\bm x)$
	for positive integers $\mu_1,\ldots,\mu_s$ with $\lambda_t+\mu_1+\cdots+\mu_s=\lambda_1+\cdots+\lambda_t$ and such that all but one of the $\lambda_t,\mu_1,\ldots,\mu_s$ are divisible by $m+1$, and the remaining one is congruent to $m$ modulo $m+1$. 
	Similarly, for each $j=1,\ldots,t-1$, we also apply the induction hypothesis to
	$m_{\lambda_1,\ldots,\lambda_j+\lambda_t,\ldots,\lambda_{t-1}}(\bm x)$,
	which again yields a linear combination of terms of the desired form. 
	Uniqueness follows from the
	linear independence of the power-sum symmetric polynomials
	$p_{\mu_1,\ldots,\mu_s}(\bm x)$, since
	$n\ge k-1\ge s$.
\end{proof}

Recall that $W_{m,k}$ is spanned by all polynomials
\[
(x_1\cdots x_n)^m (x_1^{l_1} \cdots x_n^{l_n})^{m+1} p_d(\bm x),
\]
where $d, l_1, \ldots, l_n$ are nonnegative integers satisfying
$
d + (m+1)(l_1+\cdots+l_n) = (m+1)(k-1)-1,
$
as in \eqref{formula1}.

\begin{Lemma}\label{c32}
    Let $k\ge 2$ and $n\ge k-1$. Let $P\in V_{m,k}$ be a symmetric polynomial and let $\overline{P}$ be the homogeneous component of $P$ of degree $mn+(m+1)(k-1)-1$. Suppose that $\overline{P}$ is divisible by $(x_1\cdots x_n)^m$. Then we have $\overline{P}\in W_{m,k}$. 
    Furthermore, when we express $\overline{P}$ in terms of the basis of $W_{m,k}$ given in \eqref{formula1}, then the coefficient of 
    \[
    (x_1\cdots x_n)^m 
    p_{(m+1)(k-1)-1}(\bm x)
    \]
    is equal to the coefficient of $p_{(m+1)(k-1)-1}(\bm x)$ when writing the symmetric polynomial $\widehat P =\frac{\overline{P}}{(x_1\cdots x_n)^m}$ in terms of power-sum symmetric polynomials.
\end{Lemma}
\begin{proof}
    Recall that any monomial of $P\in V_{m,k}$ must be of the form 
    \[
       x_1^{(m+1)l_1+r_1} \cdots x_n^{(m+1)l_n+r_n},
    \]
    where $l_1,\ldots,l_n$ are nonnegative integers with $\sum_{i=1}^{n}l_i<k$ and $r_1,\ldots,r_n\in \left\{0,1,\ldots,m\right\}$. If $(m+1)\sum_{i=1}^{n}l_i+\sum_{i=1}^{n}r_i=mn+(m+1)(k-1)-1$, then necessarily $\sum_{i=1}^{n}l_i=k-1$ and $\sum_{i=1}^{n}r_i=mn-1$.
    Note that $\overline{P}$ denotes the homogeneous component of $P$ of degree $mn+(m+1)(k-1)-1$, so every monomial of $\overline P$ has exactly $n-1$ variables whose exponents are congruent to $m$ modulo $m+1$, and the remaining variable has exponent congruent to $m-1$ modulo $m+1$.
    Consequently, every monomial of $\widehat P$ has exactly $n-1$ variables whose exponents are divisible by $m+1$, and the remaining variable has exponent congruent to $m$ modulo $m+1$. Moreover, $\widehat P$ is homogeneous of degree $(m+1)(k-1)-1$.
     
    Since $P$ is symmetric, both $\overline{P}$ and $\widehat P$ are symmetric as well.
    Hence $\widehat P$ can be written uniquely as a linear combination of augmented monomial symmetric polynomials
    $m_{\lambda_1,\ldots,\lambda_t}(\bm x)$,
    where $\lambda_1+\cdots+\lambda_t=(m+1)(k-1)-1$, such that all but one of the $\lambda_i$ are divisible by $m+1$, and the remaining one is congruent to $m$
    modulo $m+1$.
    Applying Lemma \ref{c31} to each augmented monomial symmetric polynomial appearing in $\widehat P$, we conclude that $\widehat P$ has a unique expression as a linear combination of terms of the form
    $p_{\mu_1,\ldots,\mu_s}(\bm x)$,
    where $\mu_1,\ldots,\mu_s$ are positive integers with $\mu_1+\cdots+\mu_s=(m+1)(k-1)-1$, such that all but one of the $\mu_j$ are divisible by $m+1$, and the remaining one is congruent to $m$ modulo $m+1$. Without loss of generality, we assume 
    $\mu_1 \equiv m \pmod{m+1}$, and $\mu_2, \ldots, \mu_s$ are divisible by $m+1$.
    Multiplying by $(x_1\cdots x_n)^m$, we obtain that $\overline{P}=(x_1\cdots x_n)^m\cdot \widehat P$ can be written uniquely as a linear combination of terms of the form 
    \begin{equation}\label{overlineP}
    	(x_1\cdots x_n)^m p_{\mu_1,\ldots,\mu_s}(\bm x)=(x_1\cdots x_n)^m p_{\mu_1}(\bm x)p_{\mu_2,\ldots,\mu_s}(\bm x),
    \end{equation}
    where $\mu_1,\ldots,\mu_s$ are positive integers satisfying $\mu_1+\cdots+\mu_s=(m+1)(k-1)-1$, with $\mu_1 \equiv m \pmod{m+1}$ and $\mu_2, \ldots, \mu_s$ divisible by $m+1$.
    By convention, $p_{\mu_2,\ldots,\mu_s}(\bm x)=1$ when $s=1$.
    Expanding the factor
    $p_{\mu_2,\ldots,\mu_s}(\bm x)$ yields only monomials in which the exponent of each variable is divisible by $m+1$. Hence each term of the form \eqref{overlineP} can be expressed as a linear combination of terms of the form \eqref{formula1} with $d=\mu_1$. Therefore,
    \[
    \overline{P}\in W_{m,k}.
    \]
    Finally, let $\lambda$ denote the coefficient of
    $p_{(m+1)(k-1)-1}(\bm x)$
    in the expansion of $\widehat P$ in terms of power-sum symmetric polynomials. When expressing $\overline P$ in terms of the form \eqref{overlineP}, the only term with $\mu_1=(m+1)(k-1)-1$ is
    \begin{equation*}
    (x_1\cdots x_n)^m
    p_{(m+1)(k-1)-1}(\bm x),
    \end{equation*}
    and its coefficient is exactly $\lambda$.  Moreover, when $\overline P$ is expressed in the basis of $W_{m,k}$ given in \eqref{formula1}, the polynomial above is the unique basis element with $d=(m+1)(k-1)-1$, and its coefficient is also equal to $\lambda$.
\end{proof}

We now prove Lemma \ref{p1}. The proof relies on two claims, whose proofs are postponed to the next two sections.

\begin{proof}[\bf {Proof of Lemma \ref{p1}}]
    Fix $k\ge 2$ and $n\ge k-1$. 
    We construct a symmetric polynomial $P_0\in U_{m,k}$ with the following two desired properties. Its homogeneous component $\varphi_{m,k}(P_0)$ of degree   
    $mn+(m+1)(k-1)-1$ belongs to $W_{m,k}$.  
    Moreover, when $\varphi_{m,k}(P_0)$ is expressed in the basis of $W_{m,k}$ 
    given in \eqref{formula1}, the coefficient of
    \[
    (x_1\cdots x_n)^m
    p_{(m+1)(k-1)-1}(\bm x)
    \]
    is nonzero.
    Define the symmetric polynomial
    \begin{equation}\label{for2}
        P(\bm x)=\prod_{i=1}^{n}\bigl((x_i-1)^{\underline{m}}\bigr)^k.
    \end{equation}
    Then $P$ has zeros of multiplicity at least $k$ at all points in $mB^n\setminus\{\bm{0}\}$.   
    By Corollary \ref{coro3.4}, there exists a unique polynomial $P_0\in V_{m,k}$ such that $P-P_0$ has zeros of multiplicity at least $k$ at all points in $mB^n \setminus \{\bm{0}\}$ and a zero of multiplicity at least $k-1$ at $\bm{0}$. It follows that $P_0$ also has zeros of multiplicity at least $k$ at all points in $mB^n\setminus\{\bm{0}\}$, and hence $P_0\in U_{m,k}$. Since $P$ is symmetric and $P_0$ is unique, $P_0$ must also be symmetric. Otherwise, permuting the variables of $P_0$ would produce another polynomial satisfying the same properties, contradicting uniqueness.
    To apply Lemma \ref{c32} to $P_0$, we need the following claim, whose proof will be given in Section \ref{Sec6}.
    \begin{claim}\label{l33}
        The polynomial $\varphi_{m,k}(P_0)$ is divisible by $(x_1\cdots x_n)^m$.
    \end{claim}
    \noindent Combining Claim \ref{l33} with Lemma \ref{c32}, we conclude that $\varphi_{m,k}(P_0)\in W_{m,k}$. Moreover, Lemma \ref{c32} implies that when $\varphi_{m,k}(P_0)$ is expressed in the basis of $W_{m,k}$, the coefficient of 
    $(x_1\cdots x_n)^m p_{(m+1)(k-1)-1}(\bm x)$ 
    equals the coefficient of 
    $p_{(m+1)(k-1)-1}(\bm x)$
    when $\frac{\varphi_{m,k}(P_0)}{(x_1\cdots x_n)^m}$ is written in terms of power-sum symmetric polynomials. Thus, it remains to show that the latter coefficient is nonzero, as stated in Claim \ref{l34}, which is proved in Section \ref{Sec7}.
    \begin{claim}\label{l34}
        When writing $\frac{\varphi_{m,k}(P_0)}{(x_1\cdots x_n)^m}$ in terms of power-sum symmetric polynomials, the coefficient of $p_{(m+1)(k-1)-1}(\bm x)$ is nonzero.
    \end{claim}
    \noindent This completes the proof of Lemma \ref{p1}.
\end{proof}

\section{Divisibility of the Highest-Degree Component}\label{Sec6}
In this section, we establish Claim \ref{l33} by studying the polynomial $P_0$ and its homogeneous component $\varphi_{m,k}(P_0)$ of degree $mn+(m+1)(k-1)-1$. Fix $k\ge 2$ and $n\ge k-1$. Recall from \eqref{for2} that
\begin{equation*}
    P(\bm x)=\prod_{i=1}^{n}\bigl((x_i-1)^{\underline{m}}\bigr)^k.
\end{equation*}

\begin{Lemma}\label{c45}
For $k\ge 1$, the polynomial $\bigl((x_i-1)^{\underline{m}}\bigr)^k$ has a unique expansion of the form
\[
\bigl((x_i-1)^{\underline{m}}\bigr)^k =a_{k,0,m}(x_{i}-1)^{\underline{m}}
    +\sum_{l=1}^{q_k} \sum_{r = 0}^{m} a_{k,l,r} \bigl((x_i)^{\underline{m+1}}\bigr)^l
    (x_i - m)^{\overline{r}},
\]
where $q_k=\lfloor mk/(m+1)\rfloor$. Moreover,
$a_{k,l,r}=0$ whenever $(m+1)l+r>mk$.
\end{Lemma}

\begin{proof}
	The polynomials $(x_i-1)^{\underline s}$ for
	$s=0,1,\ldots,m$ and
	\begin{equation*}
		\bigl((x_i)^{\underline{m+1}}\bigr)^l
		(x_i-m)^{\overline r},
		\qquad
		1\le l\le q_k,\quad 0\le r\le m,
	\end{equation*}
	have pairwise distinct degrees, namely
	$0,1,\ldots,(m+1)q_k+m$.
	Since $(m+1)q_k+m \ge mk$, the polynomial
	$\bigl((x_i-1)^{\underline{m}}\bigr)^k$ of degree $mk$
	has a unique expansion as a linear combination of these polynomials.
	Moreover, $\bigl((x_i-1)^{\underline{m}}\bigr)^k$,
	$(x_i-1)^{\underline{m}}$, and
	$\bigl((x_i)^{\underline{m+1}}\bigr)^l
	(x_i-m)^{\overline{r}}$
	all vanish at the points $x_i=1,2,\ldots,m$. Hence the remaining
	linear combination of $(x_i-1)^{\underline s}$ for $0\le s<m$
	has degree at most $m-1$ and has $m$ distinct roots. It is therefore
	the zero polynomial, which gives the stated expansion. Since the
	left-hand side has degree $mk$, no term of degree greater than $mk$
	can occur with a nonzero coefficient in this expansion.
\end{proof}

By Lemma \ref{c45}, the polynomial $P(\bm x)$ can be written as
\[
\prod_{i=1}^n \Bigl(
a_{k,0,m}(x_{i}-1)^{\underline{m}}
+\sum_{l=1}^{q_k} \sum_{r = 0}^{m} a_{k,l,r} \bigl((x_i)^{\underline{m+1}}\bigr)^l
(x_i - m)^{\overline{r}}
\Bigr),
\]
where $q_k=\lfloor\frac{mk}{m+1}\rfloor$.
Note that $(x_i-1)^{\underline m}=(x_i-m)^{\overline m}$. Expanding $P(\bm x)$ over all $i$, each resulting term is of the form
\begin{equation*}
	\prod_{i=1}^n \bigl((x_i)^{\underline{m+1}}\bigr)^{l_i}
	(x_i - m)^{\overline{r_i}}, \quad 0 \le l_i\le q_k,\; 0\le r_i\le m
\end{equation*}
with $r_i=m$ whenever $l_i=0$.
Recall that $P_0$ is the unique polynomial in $V_{m,k}$ such that $P-P_0$ has zeros of multiplicity at least $k$ at all points in $mB^n \setminus \{\bm{0}\}$ and a zero of multiplicity at least $k-1$ at $\bm{0}$. Using the above expression of $P$, we now analyze $P_0$.
\begin{Lemma}\label{P_1}
	Let $P_1$ denote the polynomial consisting of those terms in the expansion of $P$ with $\sum_{i=1}^n l_i \leq k-1$ and of degree at most $mn + (m+1)(k-1)-1$. Then $P_1=P_0$.
\end{Lemma}
\begin{proof}
	By definition, $P_1$ consists of all terms in the expansion of $P$ of the form
	\begin{equation*}
		\prod_{i=1}^n \bigl((x_i)^{\underline{m+1}}\bigr)^{l_i}
		(x_i - m)^{\overline{r_i}},\quad 0\le r_i\le m
	\end{equation*}
	with $\sum_{i=1}^n l_i \leq k-1$ and of degree at most $mn + (m+1)(k-1)-1$. The leading monomials of these terms are
	\begin{equation*}
		\prod_{i=1}^n x_i^{(m+1){l_i}+r_i},\quad 0 \le r_i\le m.
	\end{equation*}
	Hence,
	no monomial of $P_1$ is divisible by $x_{i_1}^{m+1}\cdots x_{i_k}^{m+1}$, and all monomials of $P_1$ have degree at most $mn + (m+1)(k-1) - 1$. This means $P_1\in V_{m,k}$. Now we consider $P-P_1$. Note that each term in the expansion of $P$ is divisible by
	\[
	\prod\limits_{i=1}^{n}\bigl((x_i)^{\underline{m+1}}\bigr)^{l_i}.
	\]
	Any term with $\sum_{i=1}^{n}l_i\ge k$ has a zero of multiplicity at least $k$ at every point in $mB^n$. For terms with $\sum_{i=1}^n l_i \le k-1$, the total degree is at most $mn + (m+1)(k-1)$. Among these, the terms with degree exactly $mn + (m+1)(k-1)$ must be of the form
	\begin{equation*}
	    \prod_{i=1}^{n}\bigl((x_i)^{\underline{m+1}}\bigr)^{l_i}(x_i-m)^{\overline{m}}
	\end{equation*}
	with $\sum_{i=1}^{n}l_i= k-1$. Such terms vanish with multiplicity at least $k$ at every point of $mB^n \setminus \{\bm{0}\}$ and with multiplicity at least $k-1$ at $\bm{0}$.   
	Consequently, $P-P_1$ has zeros of multiplicity at least $k$ at every point of $mB^n \setminus \{\bm{0}\}$ and a zero of multiplicity at least $k-1$ at $\bm{0}$. By Corollary \ref{coro3.4}, it follows that $P_1 = P_0$.
\end{proof}

\begin{proof}[\bf{Proof of Claim \ref{l33}}]
	By definition, $\varphi_{m,k}(P_0)$ denotes the homogeneous component of $P_0$ of degree $mn + (m+1)(k-1)-1$.
	By Lemma \ref{P_1}, this is exactly the homogeneous component of $P_1$ of that degree.
	Recall that
	\[
	P =\prod_{i=1}^n \Bigl(
	a_{k,0,m}(x_{i}-1)^{\underline{m}}
	+\sum_{l=1}^{q_k} \sum_{r = 0}^{m} a_{k,l,r} \bigl((x_i)^{\underline{m+1}}\bigr)^l
	(x_i - m)^{\overline{r}}
	\Bigr),
	\]
	where $q_k=\lfloor\frac{mk}{m+1}\rfloor$, and $P_1$ consists of all terms in the expansion of $P$ of the form
	\begin{equation*}
		\prod_{i=1}^n \bigl((x_i)^{\underline{m+1}}\bigr)^{l_i}
		(x_i - m)^{\overline{r_i}},\quad 0\le r_i\le m
	\end{equation*}
	with $\sum_{i=1}^n l_i \leq k-1$ and of degree at most $mn + (m+1)(k-1)-1$.
	Taking leading monomials and retaining only those of degree $mn + (m+1)(k-1)-1$, we obtain $\varphi_{m,k}(P_0)$ by expanding
	\begin{equation}\label{for6}
		\prod_{i=1}^{n}
		\Bigl(a_{k,0,m}x_{i}^m
		+\sum_{l=1}^{q_k} \sum_{r = 0}^{m} a_{k,l,r} x_i^{(m+1)l+r}\Bigr)
	\end{equation}
	and keeping those monomials
	\begin{equation*}
		\prod_{i=1}^n x_i^{(m+1)l_i+r_i},\quad 0\le r_i\le m,\quad \sum_{i=1}^n l_i \le k-1
	\end{equation*}
	of total degree exactly $mn + (m+1)(k-1)-1$.
	Observe that such a term has degree $mn + (m+1)(k-1) - 1$ only if it consists of $n-1$ factors of the form $x_i^{(m+1)l_i + m}$ and one factor of the form $x_j^{(m+1)l_j + m - 1}$, with $\sum_{i=1}^{n}l_i = k-1$. Hence, $\varphi_{m,k}(P_0)$ is obtained by expanding \eqref{for6} and retaining only those monomials of degree exactly $mn + (m+1)(k-1) - 1$ in which all but one of the exponents are congruent to $m$ modulo $m+1$, and the remaining one is congruent to $m-1$ modulo $m+1$.
	By Lemma \ref{c45}, each factor in \eqref{for6} has degree at most
	$mk$. We may therefore introduce coefficients $b_0,\ldots,b_{(k-1)m}$ and write \eqref{for6} as
	\[
	\prod_{i=1}^{n}\Bigl(b_{(k-1)m}x_i^{km}+b_{(k-1)m-1}x_i^{km-1}+\cdots+b_{0}x_i^m\Bigr),
	\]
	and define $b_d = 0$ for $d \ge (k-1)m + 1$.
	Note that $a_{k,l,r}=b_{(m+1)l+r-m}$. With this notation, we obtain that
	\begin{equation}\label{for7}
		\varphi_{m,k}(P_0)=\sum_{\left(\lambda_1,\ldots,\lambda_n\right)}b_{\lambda_1}\cdots b_{\lambda_n}\cdot x_1^{\lambda_1+m}\cdots x_n^{\lambda_n+m},
	\end{equation}
	where the sum is over all $n$-tuples
	$(\lambda_1,\ldots,\lambda_n)$ of nonnegative integers with $\sum_{i=1}^n (\lambda_i+m) = mn + (m+1)(k-1)-1$ such that all but one of the $\lambda_i+m$ are congruent to $m$ modulo $m+1$, and the remaining one is congruent to $m-1$ modulo $m+1$. Equivalently, $\sum_{i=1}^n \lambda_i=(m+1)(k-1)-1$, with all but one of the $\lambda_i$ divisible by $m+1$ and the remaining one congruent to $m$ modulo $m+1$.
	Since all monomials in \eqref{for7} are divisible by $(x_1\cdots x_n)^m$, so is $\varphi_{m,k}(P_0)$. 
\end{proof}

\section{Coefficient Nonvanishing and Lagrange Inversion}\label{Sec7}
To prove Claim \ref{l34}, we compute the coefficient $Y_{m,k}$ of
$p_{(m+1)(k-1)-1}(\bm x)$ in the expansion of $\frac{\varphi_{m,k}(P_0)}{(x_1\cdots x_n)^m}$ in terms of power-sum symmetric polynomials. 
We begin by establishing the following lemma.

\begin{Lemma}\label{lem7.1}
	For any sequence $\left(\lambda_1,\ldots,\lambda_t\right)$ of positive integers with $t\le n$ and $\lambda_1+\cdots+\lambda_t \le (m+1)(k-1)-1$, when expressing
	$m_{\lambda_1,\ldots,\lambda_t}(\bm x)$
	in terms of power-sum symmetric polynomials, the coefficient of $p_{\lambda_1+\cdots+\lambda_t}(\bm x)$ is $(-1)^{t-1}(t-1)!$.
\end{Lemma}
\begin{proof}
	We proceed by induction on $t$. The base case $t=1$ is immediate.
	Suppose that the statement holds for $t-1$. Recall that
	\begin{equation*}
		m_{\lambda_1,\ldots,\lambda_t}(\bm x)=p_{\lambda_t}(\bm x)m_{\lambda_1,\ldots,\lambda_{t-1}}(\bm x) -\sum_{j=1}^{t-1}m_{\lambda_1,\ldots,\lambda_j+\lambda_t,\ldots,\lambda_{t-1}}(\bm x).
	\end{equation*}
	We now expand the terms on the right-hand side in terms of power-sum symmetric polynomials.
	All terms arising from the first part contain 
	$p_{\lambda_t}(\bm x)$ as a factor, and hence do not contribute to the term
	$p_{\lambda_1+\cdots+\lambda_t}(\bm x)$.
	For the second part, by the induction hypothesis, the coefficient of $p_{\lambda_1+\cdots+\lambda_t}(\bm x)$ is
	\begin{equation*}
		-\sum_{j=1}^{t-1}(-1)^{t-2}(t-2)! =(-1)^{t-1}(t-1)!.
	\end{equation*}
	This completes the proof.
\end{proof}

Applying Lemma \ref{lem7.1}, we obtain the following expression for $Y_{m,k}$.
\begin{Lemma}\label{Ymk0}
	The coefficient $Y_{m,k}$ can be expressed as
	\begin{equation*}
		Y_{m,k}= \sum_{\left(l_1,\ldots,l_t\right)} (-1)^{t-1} a_{k,0,m}^{n-t} a_{k,l_1,m-1}a_{k,l_2,m}\cdots a_{k,l_t,m},
	\end{equation*}
    where the sum is over all sequences $\left(l_1,\ldots,l_t\right)$ of positive integers with $l_1+\cdots+l_t=k-1$.
\end{Lemma}

\begin{proof}
    From \eqref{for7}, we have
    \begin{equation*}
    	\frac{\varphi_{m,k}(P_0)}{(x_1\cdots x_n)^m}=\sum_{\left(\lambda_1,\ldots,\lambda_n\right)}b_{\lambda_1}\cdots b_{\lambda_n}\cdot x_1^{\lambda_1}\cdots x_n^{\lambda_n},
    \end{equation*}
    where the sum is over all $n$-tuples $(\lambda_1,\ldots,\lambda_n)$ of nonnegative integers satisfying $\lambda_1+\cdots+\lambda_n=(m+1)(k-1)-1$, with exactly one $\lambda_i\equiv m \pmod{m+1}$ and all others divisible by $m+1$.
    Let $t$ denote the number of positive entries among $(\lambda_1,\ldots,\lambda_n)$. We can rewrite the above equation as 
    \begin{equation}\label{for8}
    	\frac{\varphi_{m,k}(P_0)}{(x_1\cdots x_n)^m}
    	=\sum_{\left(\lambda_1,\ldots,\lambda_t\right)}b_0^{n-t}b_{\lambda_1}\cdots b_{\lambda_t}\sum_{1\le i_1<\cdots<i_t\le n} x_{i_1}^{\lambda_1}\cdots x_{i_t}^{\lambda_t},
    \end{equation}
    where the sum is over all sequences $\left(\lambda_1,\ldots,\lambda_t\right)$ of positive integers satisfying $\lambda_1+\cdots+\lambda_t=(m+1)(k-1)-1$, with exactly one $\lambda_i\equiv m \pmod{m+1}$ and all others divisible by $m+1$. 
    For each fixed $t$, the set of such $t$-tuples
    $(\lambda_1,\ldots,\lambda_t)$ is invariant under the action of the
    symmetric group $S_t$, and the coefficient
    $b_0^{n-t}b_{\lambda_1}\cdots b_{\lambda_t}$ is invariant as well.
    Averaging \eqref{for8} over all $t!$ permutations of the positions therefore gives
    \begin{equation}\label{for9}
    	\frac{\varphi_{m,k}(P_0)}{(x_1\cdots x_n)^m}
		=\sum_{(\lambda_1,\ldots,\lambda_t)} \frac{b_0^{n-t} b_{\lambda_1}\cdots b_{\lambda_t}}{t!}m_{\lambda_1,\ldots,\lambda_t}(\bm x).
    \end{equation}
    Applying Lemma \ref{lem7.1} to the terms on the right-hand
    side of \eqref{for9} yields
    \begin{equation*}
    	Y_{m,k}=\sum_{\left(\lambda_1,\ldots,\lambda_t\right)}(-1)^{t-1}(t-1)! \cdot  \frac{b_0^{n-t} b_{\lambda_1}\cdots b_{\lambda_t}}{t!}
    	=\sum_{\left(\lambda_1,\ldots,\lambda_t\right)} (-1)^{t-1}\cdot \frac{b_0^{n-t} b_{\lambda_1}\cdots b_{\lambda_t}}{t}.
    \end{equation*}
    Note that the $t$ cyclic permutations of a fixed $t$-tuple
    $(\lambda_1,\ldots,\lambda_t)$ contribute equally to the sum above,
    and exactly one of them has first entry congruent to $m$ modulo
    $m+1$. Therefore,
    \begin{equation*}
    	Y_{m,k}=\sum_{\left(\lambda_1,\ldots,\lambda_t\right)}  (-1)^{t-1} \, t \cdot \frac{b_0^{n-t} b_{\lambda_1}\cdots b_{\lambda_t}}{t} 
    	=\sum_{\left(\lambda_1,\ldots,\lambda_t\right)}(-1)^{t-1}b_0^{n-t} b_{\lambda_1}\cdots b_{\lambda_t},
    \end{equation*}
    where the sums are over all sequences $\left(\lambda_1,\ldots,\lambda_t\right)$ of positive integers with $\lambda_1+\cdots+\lambda_t=(m+1)(k-1)-1$ such that $\lambda_1\equiv m \pmod{m+1}$ and $\lambda_2,\ldots,\lambda_t$ are divisible by $m+1$.
    Writing $\lambda_1=(m+1)l_1-1$ and $\lambda_j=(m+1)l_j$ for $2\le j\le t$, we obtain
    \begin{equation*}
    	Y_{m,k}=\sum_{\left(l_1,\ldots,l_t\right)}(-1)^{t-1} b_0^{n-t} b_{(m+1)l_1-1}b_{(m+1)l_2}\cdots b_{(m+1)l_t},
    \end{equation*}
    where the sum is over all sequences $\left(l_1,\ldots,l_t\right)$ of positive integers with $l_1+\cdots+l_t=k-1$.
    Using the relation $a_{k,l,r}=b_{(m+1)l+r-m}$, we obtain the desired expression for $Y_{m,k}$.	
\end{proof}

By Lemma \ref{c45}, we have
	\begin{equation}\label{id}
		\bigl((x-1)^{\underline m}\bigr)^{k}
		= a_{k,0,m}(x-1)^{\underline m}
		+ \sum_{l=1}^{q_{k}} \sum_{r=0}^{m} a_{k,l,r}
		\bigl((x)^{\underline{m+1}}\bigr)^l
		(x - m)^{\overline{r}},
	\end{equation}
		where $q_{k}=\lfloor \frac{mk}{m+1}\rfloor$.
For the recurrences below, extend the notation $a_{k,l,r}$
to all integer indices $l,r$ by setting $a_{k,l,r}=0$ whenever
$l<0$, $r\notin\{0,\ldots,m\}$, $l=0$ and $r<m$, or
$(m+1)l+r>mk$.
\begin{Lemma}\label{lem7.3}
	 For $k\ge2$ and $0\le l\le k-1$, the coefficients $a_{k,l,r}$ satisfy
	\begin{align*}
		a_{k,l,r}&=a_{k-1,l-1,r+1}-(m-r-1)a_{k,l,r+1}, \quad 0\le r \le m-1,\\
		a_{k,l,m}&=a_{k-1,l,0}-ma_{k,l+1,0}.
	\end{align*}
\end{Lemma}

\begin{proof}
	Using \eqref{id} with $k-1$ in place of $k$, we obtain
	\begin{equation*}
		\bigl((x-1)^{\underline m}\bigr)^{k-1}
		= a_{k-1,0,m}(x-1)^{\underline m}
		+ \sum_{l=1}^{q_{k-1}} \sum_{r=0}^{m} a_{k-1,l,r}
		\bigl((x)^{\underline{m+1}}\bigr)^l
		(x - m)^{\overline{r}}.
	\end{equation*}
    Multiplying both sides by $(x-1)^{\underline m}$ gives
	\begin{equation}\label{eqk}
			\bigl((x-1)^{\underline m}\bigr)^k
			= a_{k-1,0,m}\bigl((x-1)^{\underline m}\bigr)^2  + \sum_{l=1}^{q_{k-1}}\sum_{r=0}^{m} a_{k-1,l,r}
			\bigl((x)^{\underline{m+1}}\bigr)^l
			(x-1)^{\underline m}
			(x - m)^{\overline{r}}.
	\end{equation}
    For $1\le r\le m$, we have
	\begin{equation*}
		(x-1)^{\underline m}(x-m)^{\overline r}
		=(x)^{\underline{m+1}}(x-m)^{\overline{r-1}}
		-(m+1-r)(x-1)^{\underline m}(x-m)^{\overline{r-1}}.
	\end{equation*}
    By applying this identity to the term $(x-1)^{\underline m}(x - m)^{\overline{r-1}}$ and repeating the process, we eventually express $(x-1)^{\underline m}(x - m)^{\overline{r}}$ as
    \begin{equation*}
		\begin{aligned}
		\sum_{j=0}^{r-1}
			(-1)^j
			\frac{(m-r+j)!}{(m-r)!}(x)^{\underline{m+1}}
			(x - m)^{\overline{r-1-j}}
			+(-1)^r\frac{m!}{(m-r)!}
			(x-1)^{\underline m}.
		\end{aligned}
	\end{equation*}
    Substituting this expression into \eqref{eqk} and comparing coefficients with \eqref{id} gives
    \begin{align*}
	    a_{k,l,r}&=\sum_{t=r+1}^{m}(-1)^{t-1-r}\frac{(m-1-r)!}{(m-t)!}a_{k-1,l-1,t},\quad 0\le r\le m-1,\\
	    a_{k,l,m}&=a_{k-1,l,0}+\sum_{t=1}^{m}(-1)^{t}\frac{m!}{(m-t)!}a_{k-1,l,t},
    \end{align*}
    which are equivalent to the desired recurrence relations.
\end{proof}

The following expression for $a_{k,l,r}$ satisfies the recurrence relations in Lemma \ref{lem7.3} together with the initial condition $a_{1,0,m}=1$. Denote by $[x^j]F(x)$ the coefficient of $x^j$ in the formal power series $F(x)$, i.e., $[x^j]F(x)=a_j$ if $F(x)=\sum_{j\ge0} a_jx^j$, and adopt the convention $[x^j]F(x)=0$ for $j<0$. Define the formal derivative by $F'(x)=\sum_{j\ge 0}(j+1)a_{j+1}x^j$.
\begin{Lemma}\label{lem7.4}
	The coefficients $a_{k,l,r}$ defined in \eqref{id} admit the explicit formula
	\begin{equation*}
		a_{k,l,r} = (-1)^{mk-(m+1)l-r}(m!)^{k-l-1}[x^l]\Bigl(Q_m(x)^{k-l-1}x^{\overline{m-r}}\Bigr),
		\quad 0\le l\le k-1,\quad 0\le r\le m,
	\end{equation*}
    where $Q_m(x)=\prod_{j=1}^{m}\left(1+\frac{x}{j}\right)$.
\end{Lemma}

\begin{proof}
Let $A_{k,l,r}$ denote the right-hand side of the claimed formula
for $0\le l\le k-1$ and $0\le r\le m$, and set
$A_{k,l,r}=0$ outside this range.
To verify the first recurrence, fix $0\le l\le k-1$ and
$0\le r\le m-1$, and put $S=k-l-1$. Since
\[
x^{\overline{m-r}}
=(x+m-r-1)x^{\overline{m-r-1}},
\]
we obtain
\[
[x^{l-1}]
\bigl(Q_m(x)^Sx^{\overline{m-r-1}}\bigr)
+(m-r-1)[x^l]
\bigl(Q_m(x)^Sx^{\overline{m-r-1}}\bigr)
=[x^l]\bigl(Q_m(x)^Sx^{\overline{m-r}}\bigr).
\]
Multiplying by $(-1)^{mk-(m+1)l-r}(m!)^S$ yields
\[
A_{k,l,r}
=A_{k-1,l-1,r+1}-(m-r-1)A_{k,l,r+1}.
\]
For the second recurrence, first suppose that $0\le l\le k-2$, so that
$S\ge1$. The identity
\[
(x+m)x^{\overline m}=m!\,xQ_m(x)
\]
gives
\[
[x^l]\bigl(Q_m(x)^{S-1}x^{\overline m}\bigr)
+m[x^{l+1}]
\bigl(Q_m(x)^{S-1}x^{\overline m}\bigr)
=m![x^l]Q_m(x)^S.
\]
After multiplication by
$(-1)^{m(k-1)-(m+1)l}(m!)^{S-1}$, this becomes
\[
A_{k,l,m}=A_{k-1,l,0}-mA_{k,l+1,0}.
\]
At the remaining boundary $l=k-1$, the explicit formula gives
$A_{k,k-1,m}=(-1)^{1-k}[x^{k-1}]1=0$ because $k\ge2$, while
$A_{k-1,k-1,0}=A_{k,k,0}=0$ by definition. Hence the second
recurrence also holds at this boundary.

Finally, $A_{1,0,m}=1$, while $A_{1,0,r}=0$ for $0\le r<m$, in
agreement with the zero-extension convention for $a_{k,l,r}$. The
recurrences determine all coefficients uniquely by induction on
$k$: for fixed $k\ge2$, proceed downward in $l$ from $k-1$ to $0$,
using the second recurrence to determine the entry with $r=m$ and
then the first recurrence successively for $r=m-1,\ldots,0$. Thus
$A_{k,l,r}=a_{k,l,r}$.
\end{proof}

To derive a generating function expression for $Y_{m,k}$, we recall the following two equivalent forms of the Lagrange inversion formula (see \cite[Theorem 2.1.1]{Gessel}).
\begin{Lemma}[Gessel, 2016]\label{lagrange0}
	Let $F(x), G(x), H(x)$ be formal power series with $G(0)\neq 0$. Then
	\begin{equation*}
		\begin{alignedat}{2}
			[z^j]F(y) \;\;&= \;\;\frac{1}{j}[x^{j-1}]\bigl(F'(x)G(x)^j\bigr), \quad j\ge 1,\\
			[z^j]\frac{H(y)}{1 - z G'(y)}\;\;& =\;\; [x^j]\bigl(H(x)G(x)^j\bigr), \quad j\ge0.
		\end{alignedat}
	\end{equation*}
	where $y=y(z)$ is defined implicitly by
	$y = z G(y)$.
\end{Lemma}

\begin{coro}\label{lagrange}
	Let $G(x),H(x)$ be formal power series with $G(0)\neq 0$. Then
	\begin{equation*}
		\sum_{j\ge 0} z^j [x^j]\bigl(H(x)G(x)^j\bigr)
		=\frac{H(y)}{1 - z G'(y)},
	\end{equation*}
	where $y=y(z)$ is defined implicitly by
	$y = z G(y)$.
\end{coro}

We are now ready to derive the following expression for $Y_{m,k}$.
\begin{Lemma}\label{Ymk}
	The coefficient $Y_{m,k}$ is given by
	\begin{equation*}
		Y_{m,k}
		=(-1)^{m(k-1)(n-1)-k}(m!)^{(n-1)(k-1)}
		\frac{1}{k-1}[x^{k-2}]Q_m(x)^{-(k-1)},
	\end{equation*}
	where $Q_m(x)=\prod_{j=1}^{m}\left(1+\frac{x}{j}\right)$.
\end{Lemma} 

\begin{proof} 
   By Lemma \ref{lem7.3}, we have $a_{k,l,m-1}=a_{k-1,l-1,m}$. 
   Substituting this into the expression in Lemma \ref{Ymk0} yields
   \begin{equation*}
   	Y_{m,k}= \sum_{t\ge 1}(-1)^{t-1}a_{k,0,m}^{n-t} \sum_{\left(l_1,\ldots,l_t\right)}  a_{k-1,l_1-1,m}a_{k,l_2,m}\cdots a_{k,l_t,m},
   \end{equation*}
   where the sum is over all sequences $\left(l_1,\ldots,l_t\right)$ of positive integers with $l_1+\cdots+l_t=k-1$.
   By Lemma \ref{lem7.4} with $r=m$, we have
   \begin{equation}\label{aklm}
   	a_{k,l,m} = (-1)^{m(k-1)-(m+1)l}(m!)^{k-l-1}
   	[x^l]Q_m(x)^{k-l-1}.
   \end{equation}
   For convenience, define
   $B_{k,l}=[x^l]Q_m(x)^{k-l-1}$ for all $l\ge0$. This is well-defined
   because $Q_m(0)=1$, so $Q_m(x)$ is invertible as a formal power
   series.
   Substituting \eqref{aklm} into the expression for $Y_{m,k}$, we obtain
   \begin{equation*}
   	Y_{m,k}
   	=(-1)^{m(k-1)(n-1)-k} (m!)^{(n-1)(k-1)}\sum_{t\ge 1}(-1)^{t-1} \sum_{\left(l_1,\ldots,l_t\right)} B_{k-1,l_1-1}B_{k,l_2}\cdots B_{k,l_t},
   \end{equation*}
   where the sum is over all sequences $\left(l_1,\ldots,l_t\right)$ of positive integers with $l_1+\cdots+l_t=k-1$.
   To evaluate this sum, define two ordinary generating functions by
   \begin{equation*}
   A(z)=\sum_{j\ge 0} B_{k-1,j}z^j,
   \qquad
   B(z)=\sum_{l\ge 1} B_{k,l}z^l.
   \end{equation*}
   Then $Y_{m,k}$ can be written as
   \begin{align*}
   Y_{m,k}
	&=(-1)^{m(k-1)(n-1)-k} (m!)^{(n-1)(k-1)}[z^{k-2}]\Bigl(\sum_{s\ge 0}(-1)^{s}A(z)B(z)^{s}\Bigr)\\
	&=(-1)^{m(k-1)(n-1)-k} (m!)^{(n-1)(k-1)}[z^{k-2}]
   	\frac{A(z)}{1+B(z)}.
   \end{align*}
   Let $C(z)=1+B(z)=\sum_{l\ge 0} B_{k,l}z^l$.
   Recall that $B_{k,l}=[x^l]Q_m(x)^{k-1-l}$. Thus,
   \begin{equation*}
   A(z)=\sum_{j\ge 0} z^j [x^j]\bigl(Q_m(x)^{k-2} Q_m(x)^{-j}\bigr),
   \qquad
   C(z)=\sum_{l\ge 0} z^l [x^l]\bigl(Q_m(x)^{k-1} Q_m(x)^{-l}\bigr).
   \end{equation*}
   Set $G(x)=\frac{1}{Q_m(x)}$, and let $y=y(z)$ be defined implicitly by $y=zG(y)$.
   Applying Corollary \ref{lagrange} with $H(x)=Q_m(x)^{k-2}$ and $H(x)=Q_m(x)^{k-1}$, respectively, we obtain
   \begin{equation*}
   A(z)=\frac{Q_m(y)^{k-2}}{1-zG'(y)}, 
   \qquad 
   C(z)=\frac{Q_m(y)^{k-1}}{1-zG'(y)}.
   \end{equation*}
   Hence $\frac{A(z)}{C(z)}=\frac{1}{Q_m(y)}=\frac{y(z)}{z}$.	
   Consequently,
   \[
   Y_{m,k}
   =(-1)^{m(k-1)(n-1)-k}(m!)^{(n-1)(k-1)}
   [z^{k-1}]y(z).
   \]
   Applying Lemma \ref{lagrange0} with $F(x)=x$ gives
   \begin{equation*}
   [z^{k-1}]y(z)
   =
   \frac{1}{k-1}[x^{k-2}]G(x)^{k-1}=\frac{1}{k-1}[x^{k-2}]Q_m(x)^{-(k-1)}.
   \end{equation*}
   This completes the proof.
\end{proof}

\begin{proof}[\bf{Proof of Claim \ref{l34}}]
	 Fix $k\ge 2$. By Lemma \ref{Ymk}, $Y_{m,k}\neq 0$ is equivalent to
	 \begin{equation*}
	 [x^{k-2}]Q_m(x)^{-(k-1)}\neq 0.
	\end{equation*}
	 Recall that $Q_m(x)=\prod_{j=1}^{m}\left(1+\frac{x}{j}\right)$. Then we have
	 \begin{equation*}
	 Q_m(x)^{-(k-1)}
	 =\prod_{j=1}^{m}
	 \sum_{s\ge 0}(-1)^s \binom{k+s-2}{s}\left(\frac{x}{j}\right)^s.
	\end{equation*}
	 Therefore,
	 \begin{equation*}
	 [x^{k-2}]Q_m(x)^{-(k-1)}
	 =(-1)^{k-2}\sum_{\left(s_1,\ldots,s_m\right)}\prod_{j=1}^m \binom{k+s_j-2}{s_j}\frac{1}{j^{s_j}}\neq 0,
	\end{equation*}
    where the sum is over all sequences
    $\left(s_1,\ldots,s_m\right)$ of nonnegative integers satisfying $s_1+\cdots+s_m=k-2$.
\end{proof}

\section*{Acknowledgements}
The authors would like to thank Haochen Liu for many helpful discussions. This paper is supported by the National Natural Science Foundation of China (Grant No. 12571350) and the Guangdong Basic and Applied Basic Research Foundation (Grant No. 2025A1515010457).

\appendix
\section{Explicit Constructions for \texorpdfstring{$f(2B^n,3)$}{f(2B to the n, 3)}}\label{appendix}
We present explicit constructions that determine the values of $f(2B^n,3)$ for $n=2,3,4$; see Table \ref{tab:almost-3-covers}. 
Each construction is specified by a multiset of hyperplanes. In the table,
$H^{[r]}$ denotes $r$ copies of $H$, and an unmarked hyperplane occurs once.

\begin{table}[H]
    \centering
    \caption{Explicit constructions attaining $f(2B^n,3)$ for $n=2,3,4$.}
    \label{tab:almost-3-covers}
    \vspace{2pt}
    \small
    \renewcommand{\arraystretch}{1.35}
    \begin{tabular}{@{\hspace{4pt}}c l c@{\hspace{4pt}}}
        \toprule
        $n$ & \multicolumn{1}{c}{Hyperplane multiset} & Size \\
        \midrule
        $2$
        & $\begin{array}{l}
          x_1=1,\ (x_1=2)^{[2]},\ x_2=1,\ (x_2=2)^{[2]},\\
          (x_1+x_2=1)^{[2]},\ x_1+x_2=2
          \end{array}$
        & $9$ \\
        \midrule
        $3$
        & $\begin{array}{l}
          x_1=1,\ x_1=2,\ x_2=1,\ x_2=2,\\
          x_3=1,\ x_3=2,\ x_1+x_3=2,\ x_2+x_3=2,\\
          x_1+x_2+x_3=1,\ x_1+x_2-x_3=1,\ x_1+x_2+2x_3=2
          \end{array}$
        & $11$ \\
        \midrule
        $4$
        & $\begin{array}{l}
          x_1=1,\ x_1=2,\ x_2=1,\ x_3=1,\\
          x_3=2,\ x_4=1,\ x_4=2,\ x_1+x_2=2,\\
          x_2+x_3=2,\ x_2+x_4=2,\ x_1+x_2+x_3+x_4=1,\\
          x_1-x_2+x_3+x_4=1,\ x_1+2x_2+x_3+x_4=2
          \end{array}$
        & $13$ \\
        \bottomrule
    \end{tabular}
\end{table}

\noindent Recall that $f(2B^n,3)$ denotes the minimum size of a multiset of hyperplanes that covers each point of $2B^n \setminus \{\bm{0}\}$ at least three times while leaving $\bm{0}$ uncovered.
Every hyperplane listed in the table avoids $\bm{0}$. For $n=2,3,4$,
a direct enumeration verifies that every point of
$2B^n\setminus\{\bm{0}\}$ is covered at least three times
by the corresponding multiset of hyperplanes. Thus the constructions yield
\[
f(2B^n,3)\le 2n+5\qquad (n=2,3,4).
\]
Combined with the lower bound
$f(2B^n,3)\ge 2n+5$ from Theorem \ref{main-1}, this proves
\[
f(2B^2,3)=9,\qquad f(2B^3,3)=11,\qquad f(2B^4,3)=13.
\]

\end{document}